\documentclass[11pt]{amsart}

\usepackage[foot]{amsaddr}

\usepackage{amsmath}
\usepackage{amsthm}
\usepackage{mathtools}
\usepackage{amssymb}

\newtheorem{theorem}{Theorem}
\newtheorem{lemma}{Lemma}
\newtheorem{proposition}{Proposition}
\newtheorem{claim}{Claim}
\newtheorem{definition}{Definition}
\newtheorem{conjecture}{Conjecture}

\usepackage{hyperref}
\usepackage[dvipsnames]{xcolor}
\usepackage{tikz}
\usetikzlibrary{calc}
\usetikzlibrary{shapes}

\usepackage{subcaption}
\usepackage{enumerate}
\usepackage{todonotes}
\usepackage{booktabs}
\usepackage{adjustbox}
\usepackage{graphicx}
\usepackage{ stmaryrd }
\usepackage{cleveref}

\Crefname{claim}{Claim}{Claims}
\Crefname{conjecture}{Conjecture}{Conjectures}
\Crefname{Conjecture}{Conjecture}{Conjectures}
\crefname{conjecture}{conjecture}{conjectures}

\DeclareMathOperator*{\argmax}{arg\,max}

\hypersetup{
  colorlinks = true,
  citecolor = purple,
  linkcolor = blue,
  urlcolor = blue
}

\newcommand{\csproblem}[3]{
    \begin{center}
    \fbox{\begin{tabular}{lp{8cm}}
    {\small PROBLEM:} : & #1 \\
    {\small INPUT} : & #2 
    \\ 
    {\small OUTPUT} : & #3\\
    \end{tabular}}
    \end{center}
    }

 \usepackage{xcolor}

\newcommand{\packit}{\textsc{PackIt!}}

\newcommand{\onesquare}[4]{
    \node[draw, fill=#3, opacity=1, minimum width=0.75cm, minimum height =0.75cm] (#1#2) at (#1, #2) {#4};
 }

\newcommand{\onesquareTwo}[4]{
    \node[draw, fill=#3, opacity=1, minimum width=0.6cm, minimum height =0.6cm] (#1#2) at (#1, #2) {\tiny #4};
 }

\definecolor[named]{ACMBlue}{cmyk}{1,0.1,0,0.1}
\definecolor[named]{ACMYellow}{cmyk}{0,0.16,1,0}
\definecolor[named]{ACMOrange}{cmyk}{0,0.42,1,0.01}
\definecolor[named]{ACMRed}{cmyk}{0,0.90,0.86,0}   
\definecolor[named]{ACMLightRed}{cmyk}{0,0,0,0.35}
\definecolor[named]{ACMLightBlue}{cmyk}{0.49,0.01,0,0}
\definecolor[named]{ACMGreen}{cmyk}{0.20,0,1,0.19}
\definecolor[named]{ACMPurple}{cmyk}{0.55,0.6,0.1,0.15}
\definecolor[named]{ACMPurple2}{cmyk}{0.04,0.7,0.01,0.01}
\definecolor[named]{ACMPurple3}{cmyk}{0.04,0.7,0.01,0.01}
\definecolor[named]{ACMDarkBlue}{cmyk}{1,0.58,0,0.21}

\definecolor{redorange}{rgb}{0.878431, 0.235294, 0.192157}
\definecolor{lightblue}{rgb}{0.552941, 0.72549, 0.792157}
\definecolor{clearyellow}{rgb}{0.964706, 0.745098, 0}
\definecolor{midyellow}{rgb}{0.764706, 0.645098, 0.5}
\definecolor{clearorange}{rgb}{0.917647, 0.462745, 0}
\definecolor{mildgray}{rgb}{0.54902, 0.509804, 0.47451}
\definecolor{softblue}{rgb}{0.643137, 0.858824, 0.909804}
\definecolor{bluegray}{rgb}{0.141176, 0.313725, 0.603922}
\definecolor{lightgreen}{rgb}{0.709804, 0.741176, 0}
\definecolor{redpurple}{rgb}{0.835294, 0, 0.196078}
\definecolor{midblue}{rgb}{0, 0.592157, 0.662745}
\definecolor{clearpurple}{rgb}{0.67451, 0.0784314, 0.352941}
\definecolor{browngreen}{rgb}{0.333333, 0.313725, 0.145098}
\definecolor{darkestpurple}{rgb}{0.396078, 0.113725, 0.196078}
\definecolor{greypurple}{rgb}{0.294118, 0.219608, 0.298039}
\definecolor{darkturqoise}{rgb}{0, 0.239216, 0.298039}
\definecolor{darkbrown}{rgb}{0.305882, 0.211765, 0.160784}
\definecolor{midgreen}{rgb}{0.560784, 0.6, 0.243137}
\definecolor{darkred}{rgb}{0.576471, 0.152941, 0.172549}
\definecolor{darkpurple}{rgb}{0.313725, 0.027451, 0.470588}
\definecolor{darkestblue}{rgb}{0, 0.156863, 0.333333}
\definecolor{lightpurple}{rgb}{0.776471, 0.690196, 0.737255}
\definecolor{softgreen}{rgb}{0.733333, 0.772549, 0.572549}
\definecolor{offwhite}{rgb}{0.839216, 0.823529, 0.768627}
\definecolor{brightgreen}{rgb}{0.85, 0.98, 0.01}
\definecolor{sunriseorange}{rgb}{0.98, 0.54, 0.22}
\definecolor{twilightblue}{rgb}{0.11, 0.22, 0.42}
\definecolor{forestgreen}{rgb}{0.13, 0.55, 0.13}
\definecolor{stormygray}{rgb}{0.4, 0.4, 0.47}
\definecolor{candypink}{rgb}{0.88, 0.44, 0.88}
\definecolor{lavenderpurple}{rgb}{0.7, 0.5, 0.8}
\definecolor{fieryred}{rgb}{0.8, 0.12, 0.12}
\definecolor{oceanblue}{rgb}{0.16, 0.5, 0.73}
\definecolor{goldenyellow}{rgb}{0.95, 0.82, 0.12}
\definecolor{mintgreen}{rgb}{0.24, 0.71, 0.54}
\definecolor{sunsetorange}{rgb}{0.98, 0.37, 0.22}
\definecolor{nightblue}{rgb}{0.03, 0.12, 0.26}
\definecolor{meadowgreen}{rgb}{0.48, 0.73, 0.29}
\definecolor{rosepink}{rgb}{0.96, 0.56, 0.67}
\definecolor{skyblue}{rgb}{0.53, 0.81, 0.92}
\definecolor{earthbrown}{rgb}{0.4, 0.26, 0.13}
\definecolor{icegray}{rgb}{0.6, 0.76, 0.76}
\definecolor{lemonyellow}{rgb}{1, 0.89, 0.31}
\definecolor{sapphireblue}{rgb}{0.06, 0.32, 0.73}
\definecolor{emeraldgreen}{rgb}{0.04, 0.78, 0.47}
\definecolor{rubyred}{rgb}{0.6, 0.07, 0.11}
\definecolor{violet}{rgb}{0.56, 0.0, 1.0}
\definecolor{honeygold}{rgb}{0.87, 0.72, 0.53}
\definecolor{charcoalgray}{rgb}{0.21, 0.27, 0.31}
\definecolor{peach}{rgb}{1.0, 0.8, 0.64}
\definecolor{azure}{rgb}{0.0, 0.5, 1.0}

\DeclareMathAlphabet\EuRoman{U}{eur}{m}{n}
\SetMathAlphabet\EuRoman{bold}{U}{eur}{b}{n}
\newcommand{\euler}{\EuRoman}

\newcommand{\numa}{\textcolor{white}{$\euler{1}$}}
\newcommand{\numb}{\textcolor{white}{$\euler{2}$}}
\newcommand{\numc}{\textcolor{white}{$\euler{3}$}}
\newcommand{\numd}{\textcolor{white}{$\euler{4}$}}
\newcommand{\nume}{\textcolor{white}{$\euler{5}$}}
\newcommand{\numf}{\textcolor{white}{$\euler{6}$}}
\newcommand{\numg}{\textcolor{white}{$\euler{7}$}}
\newcommand{\numh}{\textcolor{white}{$\euler{8}$}}
\newcommand{\numi}{\textcolor{white}{$\euler{9}$}}
\newcommand{\numj}{\textcolor{white}{$\euler{10}$}}
\newcommand{\numk}{\textcolor{white}{$\euler{11}$}}
\newcommand{\numl}{\textcolor{white}{$\euler{12}$}}
\newcommand{\numm}{\textcolor{white}{$\euler{13}$}}
\newcommand{\numn}{\textcolor{white}{$\euler{14}$}}
\newcommand{\numo}{\textcolor{white}{$\euler{15}$}}
\newcommand{\nump}{\textcolor{white}{$\euler{16}$}}
\newcommand{\numq}{\textcolor{white}{$\euler{17}$}}
\newcommand{\numr}{\textcolor{white}{$\euler{18}$}}
\newcommand{\nums}{\textcolor{white}{$\euler{19}$}}
\newcommand{\numt}{\textcolor{white}{$\euler{20}$}}
\newcommand{\numu}{\textcolor{white}{$\euler{21}$}}
\newcommand{\numv}{\textcolor{white}{$\euler{22}$}}
\newcommand{\numw}{\textcolor{white}{$\euler{23}$}}
\newcommand{\numx}{\textcolor{white}{$\euler{24}$}}
\newcommand{\numy}{\textcolor{white}{$\euler{25}$}}
\newcommand{\numz}{\textcolor{white}{$\euler{26}$}}
\newcommand{\numA}{\textcolor{white}{$\euler{27}$}}
\newcommand{\numB}{\textcolor{white}{$\euler{28}$}}
\newcommand{\numC}{\textcolor{white}{$\euler{29}$}}
\newcommand{\numD}{\textcolor{white}{$\euler{30}$}}
\newcommand{\numE}{\textcolor{white}{$\euler{31}$}}
\newcommand{\numF}{\textcolor{white}{$\euler{32}$}}
\newcommand{\numG}{\textcolor{white}{$\euler{33}$}}
\newcommand{\numH}{\textcolor{white}{$\euler{34}$}}
\newcommand{\numI}{\textcolor{white}{$\euler{35}$}}
\newcommand{\numJ}{\textcolor{white}{$\euler{36}$}}
\newcommand{\numK}{\textcolor{white}{$\euler{37}$}}
\newcommand{\numL}{\textcolor{white}{$\euler{38}$}}
\newcommand{\numM}{\textcolor{white}{$\euler{39}$}}
\newcommand{\numN}{\textcolor{white}{$\euler{40}$}}
\newcommand{\numO}{\textcolor{white}{$\euler{41}$}}
\newcommand{\numP}{\textcolor{white}{$\euler{42}$}}
\newcommand{\numQ}{\textcolor{white}{$\euler{43}$}}
\newcommand{\numR}{\textcolor{white}{$\euler{44}$}}
\newcommand{\numS}{\textcolor{white}{$\euler{45}$}}
\newcommand{\numT}{\textcolor{white}{$\euler{46}$}}
\newcommand{\numU}{\textcolor{white}{$\euler{47}$}}
\newcommand{\numV}{\textcolor{white}{$\euler{48}$}}
\newcommand{\numW}{\textcolor{white}{$\euler{49}$}}
\newcommand{\numX}{\textcolor{white}{$\euler{50}$}}
\newcommand{\numY}{\textcolor{white}{$\euler{51}$}}
\newcommand{\numZ}{\textcolor{white}{$\euler{52}$}}
\newcommand{\numaa}{\textcolor{white}{$\euler{53}$}}
\newcommand{\numab}{\textcolor{white}{$\euler{54}$}}
\newcommand{\numac}{\textcolor{white}{$\euler{55}$}}
\newcommand{\numad}{\textcolor{white}{$\euler{56}$}}
\newcommand{\numae}{\textcolor{white}{$\euler{57}$}}
\newcommand{\numaf}{\textcolor{white}{$\euler{58}$}}
\newcommand{\numag}{\textcolor{white}{$\euler{59}$}}
\newcommand{\numah}{\textcolor{white}{$\euler{60}$}}
\newcommand{\numai}{\textcolor{white}{$\euler{61}$}}
\newcommand{\numaj}{\textcolor{white}{$\euler{62}$}}
\newcommand{\numak}{\textcolor{white}{$\euler{63}$}}
\newcommand{\numal}{\textcolor{white}{$\euler{64}$}}
\newcommand{\numam}{\textcolor{white}{$\euler{65}$}}
\newcommand{\numan}{\textcolor{white}{$\euler{66}$}}
\newcommand{\numao}{\textcolor{white}{$\euler{67}$}}
\newcommand{\numap}{\textcolor{white}{$\euler{68}$}}
\newcommand{\numaq}{\textcolor{white}{$\euler{69}$}}
\newcommand{\numar}{\textcolor{white}{$\euler{70}$}}

\newcommand{\cola}{darkpurple} 
\newcommand{\colb}{midgreen!60!brightgreen} 
\newcommand{\colc}{midblue} 
\newcommand{\cold}{clearorange} 
\newcommand{\cole}{redorange} 
\newcommand{\colf}{clearyellow} 
\newcommand{\colg}{midyellow!80!redorange} 
\newcommand{\colh}{lightgreen} 
\newcommand{\coli}{softblue!70!lightblue} 
\newcommand{\colj}{clearpurple} 
\newcommand{\colk}{softgreen!50!lightgreen} 
\newcommand{\coll}{bluegray} 
\newcommand{\colm}{greypurple!60!darkpurple} 
\newcommand{\coln}{redpurple} 
\newcommand{\colo}{lightpurple!80!clearpurple} 
\newcommand{\colp}{mildgray} 
\newcommand{\colq}{softgreen} 
\newcommand{\colr}{green!50!black} 
\newcommand{\cols}{darkestpurple} 
\newcommand{\colt}{darkturqoise} 
\newcommand{\colu}{lightblue} 
\newcommand{\colv}{browngreen!75!midgreen} 
\newcommand{\colw}{darkred} 
\newcommand{\colx}{darkestblue} 
\newcommand{\coly}{softgreen!50!brightgreen} 
\newcommand{\colz}{darkbrown} 
\newcommand{\colA}{sunriseorange} 
\newcommand{\colB}{twilightblue} 
\newcommand{\colC}{forestgreen} 
\newcommand{\colD}{stormygray} 
\newcommand{\colE}{candypink} 
\newcommand{\colF}{lavenderpurple} 
\newcommand{\colG}{fieryred} 
\newcommand{\colH}{oceanblue} 
\newcommand{\colI}{goldenyellow} 
\newcommand{\colJ}{mintgreen} 
\newcommand{\colK}{sunsetorange} 
\newcommand{\colL}{nightblue} 
\newcommand{\colM}{meadowgreen} 
\newcommand{\colN}{rosepink} 
\newcommand{\colO}{skyblue} 
\newcommand{\colP}{earthbrown} 
\newcommand{\colQ}{icegray} 
\newcommand{\colR}{lemonyellow} 
\newcommand{\colS}{sapphireblue} 
\newcommand{\colT}{emeraldgreen} 
\newcommand{\colU}{rubyred} 
\newcommand{\colV}{violet} 
\newcommand{\colW}{honeygold} 
\newcommand{\colX}{charcoalgray} 
\newcommand{\colY}{peach} 
\newcommand{\colZ}{azure} 

\definecolor{colaa}{rgb}{0.82, 0.66, 0.22} 
\definecolor{colab}{rgb}{0.69, 0.88, 0.90} 
\definecolor{colac}{rgb}{0.94, 0.5, 0.5} 
\definecolor{colad}{rgb}{0.0, 0.42, 0.24} 
\definecolor{colae}{rgb}{0.67, 0.88, 0.69} 
\definecolor{colaf}{rgb}{0.99, 0.76, 0.80} 
\definecolor{colag}{rgb}{0.75, 0.75, 0.75} 
\definecolor{colah}{rgb}{0.62, 0.32, 0.17} 
\definecolor{colai}{rgb}{0.53, 0.67, 0.81} 
\definecolor{colaj}{rgb}{0.85, 0.65, 0.13} 
\definecolor{colak}{rgb}{0.48, 0.25, 0.0}  
\definecolor{colal}{rgb}{0.29, 0.0, 0.51}  
\definecolor{colam}{rgb}{0.82, 0.41, 0.12} 
\definecolor{colan}{rgb}{1.0, 0.84, 0.0}   
\definecolor{colao}{rgb}{0.28, 0.24, 0.55} 
\definecolor{colap}{rgb}{0.0, 0.2, 0.4}    
\definecolor{colaq}{rgb}{0.93, 0.51, 0.93} 
\definecolor{colar}{rgb}{0.0, 0.5, 0.0}    
\definecolor{colas}{rgb}{0.55, 0.0, 0.0}   
\definecolor{colat}{rgb}{0.71, 0.40, 0.11} 
\definecolor{colau}{rgb}{0.8, 0.52, 0.25}  
\definecolor{colav}{rgb}{0.64, 0.58, 0.5}  
\definecolor{colaw}{rgb}{0.96, 0.87, 0.7}  
\definecolor{colax}{rgb}{0.91, 0.84, 0.42} 
\definecolor{colay}{rgb}{0.55, 0.76, 0.29} 
\definecolor{colaz}{rgb}{0.0, 0.75, 1.0}   
\newcommand{\colaa}{colaa}
\newcommand{\colab}{colab}
\newcommand{\colac}{colac}
\newcommand{\colad}{colad}
\newcommand{\colae}{colae}
\newcommand{\colaf}{colaf}
\newcommand{\colag}{colag}
\newcommand{\colah}{colah}
\newcommand{\colai}{colai}
\newcommand{\colaj}{colaj}
\newcommand{\colak}{colak}
\newcommand{\colal}{colal}
\newcommand{\colam}{colam}
\newcommand{\colan}{colan}
\newcommand{\colao}{colao}
\newcommand{\colap}{colap}
\newcommand{\colaq}{colaq}
\newcommand{\colar}{colar}

\newcommand{\colGG}{offwhite}

\DeclareRobustCommand{\wordlify}[2]{
    {\protect
    \hspace{-15pt}
    \raisebox{-5pt}{
        \foreach \l [count=\c, evaluate=\c as \co using {#2[\c-1]}] in #1 {
            \ifnum \co=0
                \SquareBox{\l}{\cola}{2pt}
            \else
                \ifnum \co=1
                    \SquareBox{\l}{\colb}{2pt}%
                \else
                    \SquareBox{\l}{\colc}{2pt}%
                \fi
            \fi
            \hspace{-11pt}
        }
    }
    }
}

\makeatletter
\newdimen\@myBoxHeight%
\newdimen\@myBoxDepth%
\newdimen\@myBoxWidth%
\newdimen\@myBoxSize%
\DeclareRobustCommand{\SquareBox}[3]{%
     \settoheight{\@myBoxHeight}{#3}
    \settodepth{\@myBoxDepth}{#3}
    \settowidth{\@myBoxWidth}{#3}
    \pgfmathsetlength{\@myBoxSize}{max(\@myBoxWidth,(\@myBoxHeight+\@myBoxDepth))}%
    \protect\tikz \node [shape=rectangle, draw, text=white, minimum size=\@myBoxSize,
                 fill=#2, inner sep=0, outer sep=0pt] 
                 {\textcolor{white}{$\euler{#1}$}};%
}%
\makeatother

\title[\packit: Gamified Rectangle Packing]{\wordlify{{P,a,c,k,I,t,!}}{{0,1,1,1,2,2,2}}\\  \vspace{6pt} Gamified Rectangle Packing} 

\author{$^\star$Thomas Garrison, Marijn J.H. Heule, and Bernardo Subercaseaux}

\address{Carnegie Mellon University, Pittsburgh, PA 15213, USA. $^\star$Authors are sorted alphabetically.}


\begin{document}

\begin{abstract}
We present and analyze \packit, a turn-based game consisting of packing rectangles on an $n \times n$ grid. \packit~can be easily played on paper, either as a competitive two-player game or in \emph{solitaire} fashion. On the $t$-th turn, a rectangle of area $t$ or $t+1$ must be placed in the grid. In the two-player format of~\packit~whichever player places a rectangle last wins, whereas the goal in the solitaire variant is to perfectly pack the $n \times n$ grid. We analyze necessary conditions for the existence of a perfect packing over $n \times n$, then present an automated reasoning approach that allows finding perfect games of~\packit~up to $n = 50$ which includes a novel SAT-encoding technique of independent interest, and conclude by proving an NP-hardness result.
\end{abstract}

\maketitle

\vspace{-22pt}
\section{Introduction}
\label{sec:intro}
Pen-and-paper games have not only stimulated bored high school students for centuries, but also attracted the attention of mathematicians and computer scientists alike. From Tic-Tac-Toe to Conway's \emph{Sprouts}~\cite{Gardner1970}, passing through \emph{Dots and Boxes}~\cite{buchin2021dots}, Sudoku, Hangman~\cite{barbay_et_al:LIPIcs.FUN.2021.23}, and Nim~\cite{boutonNimGameComplete1901}, simple pen-and-paper games have had a long lasting impact in combinatorial game theory~(e.g., the Sprague-Grundy theorem) and have offered landmark computational challenges (e.g., Sudokus require 17 clues to have a unique solution~\cite{mcguire201316clue}). 
In this paper we introduce a new pen-and-paper game,~\packit, and explore both mathematical and computational challenges concerning it.

\subsection{Definition of~\packit}

The game proceeds by turns, and takes place over an $n \times n$ grid that we shall denote $G$. The main principle of~\packit~ is very simple: on turn $t$ (starting from $1$), a rectangle $r_t$ of area $t$ or $t+1$ must be placed into $G$ without intersecting any of the already placed rectangles. Formally, at the beginning of the game one defines the set of \emph{used cells of the grid} as $U_0 := \varnothing$. On turn $t$, the corresponding player chooses $r_t = (h_t, v_t, x_t, y_t)$, with $h_t \cdot v_t \in \{t, t+1\}$, and $0 \leq x_t, y_t < n$. Define the cells used by this rectangle as the set
\[
A_t := \{ x_t, x_{t+1}, \ldots, x_{t+h_t-1}\} \times \{ y_t, y_{t+1}, \ldots, y_{t + v_t -1}\},
\]
so that the requirement for a valid turn is that $A_t \cap U_{t-1} = \varnothing$. After a valid turn, one sets $U_t := U_{t-1} \cup A_t$.~\Cref{fig:example-game} illustrates some examples.

\begin{figure}[t]
    \centering
    \begin{subfigure}{0.48\textwidth}
    \centering
   \begin{tikzpicture}
\onesquare{ 0.0 }{ 3.75 }{\cola}{\numa}
\onesquare{ 0.75 }{ 3.75 }{\colb}{\numb}
\onesquare{ 1.5 }{ 3.75 }{\colc}{\numc}
\onesquare{ 2.25 }{ 3.75 }{\cold}{\numd}
\onesquare{ 3.0 }{ 3.75 }{\cole}{\nume}
\onesquare{ 0.0 }{ 3.0 }{\cola}{\numa}
\onesquare{ 0.75 }{ 3.0 }{\colb}{\numb}
\onesquare{ 1.5 }{ 3.0 }{\colc}{\numc}
\onesquare{ 2.25 }{ 3.0 }{\cold}{\numd}
\onesquare{ 3.0 }{ 3.0 }{\cole}{\nume}
\onesquare{ 0.0 }{ 2.25 }{\colGG}{\;}
\onesquare{ 0.75 }{ 2.25 }{\colb}{\numb}
\onesquare{ 1.5 }{ 2.25 }{\colc}{\numc}
\onesquare{ 2.25 }{ 2.25 }{\cold}{\numd}
\onesquare{ 3.0 }{ 2.25 }{\cole}{\nume}
\onesquare{ 0.0 }{ 1.5 }{\colf}{\numf}
\onesquare{ 0.75 }{ 1.5 }{\colf}{\numf}
\onesquare{ 1.5 }{ 1.5 }{\colf}{\numf}
\onesquare{ 2.25 }{ 1.5 }{\cold}{\numd}
\onesquare{ 3.0 }{ 1.5 }{\cole}{\nume}
\onesquare{ 0.0 }{ 0.75 }{\colf}{\numf}
\onesquare{ 0.75 }{ 0.75 }{\colf}{\numf}
\onesquare{ 1.5 }{ 0.75 }{\colf}{\numf}
\onesquare{ 2.25 }{ 0.75 }{\cold}{\numd}
\onesquare{ 3.0 }{ 0.75 }{\cole}{\nume}
   \end{tikzpicture}
   \caption{An \emph{imperfect} game of~\packit }
   \end{subfigure}
    \begin{subfigure}{0.48\textwidth}
    \centering
   \begin{tikzpicture}[]
\onesquare{ 0.0 }{ 3.75 }{\cola}{\numa}
\onesquare{ 0.75 }{ 3.75 }{\cola}{\numa}
\onesquare{ 1.5 }{ 3.75 }{\colb}{\numb}
\onesquare{ 2.25 }{ 3.75 }{\cold}{\numd}
\onesquare{ 3.0 }{ 3.75 }{\cole}{\nume}
\onesquare{ 0.0 }{ 3.0 }{\colc}{\numc}
\onesquare{ 0.75 }{ 3.0 }{\colc}{\numc}
\onesquare{ 1.5 }{ 3.0 }{\colb}{\numb}
\onesquare{ 2.25 }{ 3.0 }{\cold}{\numd}
\onesquare{ 3.0 }{ 3.0 }{\cole}{\nume}
\onesquare{ 0.0 }{ 2.25 }{\colc}{\numc}
\onesquare{ 0.75 }{ 2.25 }{\colc}{\numc}
\onesquare{ 1.5 }{ 2.25 }{\colb}{\numb}
\onesquare{ 2.25 }{ 2.25 }{\cold}{\numd}
\onesquare{ 3.0 }{ 2.25 }{\cole}{\nume}
\onesquare{ 0.0 }{ 1.5 }{\colf}{\numf}
\onesquare{ 0.75 }{ 1.5 }{\colf}{\numf}
\onesquare{ 1.5 }{ 1.5 }{\colf}{\numf}
\onesquare{ 2.25 }{ 1.5 }{\cold}{\numd}
\onesquare{ 3.0 }{ 1.5 }{\cole}{\nume}
\onesquare{ 0.0 }{ 0.75 }{\colf}{\numf}
\onesquare{ 0.75 }{ 0.75 }{\colf}{\numf}
\onesquare{ 1.5 }{ 0.75 }{\colf}{\numf}
\onesquare{ 2.25 }{ 0.75 }{\cold}{\numd}
\onesquare{ 3.0 }{ 0.75 }{\cole}{\nume}

   \end{tikzpicture}
   \caption{A \emph{perfect} game of~\packit }
   \end{subfigure}
    \caption{Illustration of a couple of games of~\packit. Each rectangle $a_t$ is labeled with $t$ and depicted in a different color.}
    \label{fig:example-game}
\end{figure}

\subparagraph*{\textbf{\packit~as a game}.}
\packit~can be played as a \emph{solitaire} game, where the goal of the game is to complete a \emph{perfect packing}, that is, to play so that after a valid sequence of turns it holds that $U_t = \{0, \ldots, n-1 \} \times \{0, \ldots, n-1\}$. As depicted in~\Cref{fig:example-game}, we say the final board of such a game corresponds to a \emph{perfect game of \packit}. For two players, it suffices to alternate turns and when a player cannot play a valid turn, he or she is declared the \emph{loser}.
At this point, we suggest the reader to directly experiment with~\packit. A version of the game is available for \emph{solitaire} mode at \href{https://packit.surge.sh}{https://packit.surge.sh}.\\

\subparagraph*{\textbf{Organization}}
The main question about~\packit~is:
\begin{center}
    \emph{for which values of $n$ the $n\times n$ grid admits a perfect game of~\packit?}
\end{center}
\Cref{sec:results}~presents arithmetic results that represent the initial steps toward answering this question. Then,~\Cref{sec:complexity}~discusses the complexity of~\packit, showing that a particular version of the solitaire game is NP-hard.\Cref{sec:computation} is devoted to the problem of computationally generating perfect games of~\packit. We present a novel SAT encoding that allows us to obtain solutions up to $n \leq 50$, thereby providing partial evidence for the main conjecture presented in~\Cref{sec:results}.

\section{Arithmetic Results}
\label{sec:results}
A perfect game of~\packit~can be conceptually divided into two aspects: 
\begin{itemize}
    \item \textbf{(Rectangle selection)} As we denote by $|A_t|$ the area of the rectangle used in turn $t$, it must hold that in a perfect game of~\packit~we have
    \[
    \sum_{t} |A_t| = n^2.
    \]
    Moreover, in order to fit every rectangle $r_t$ of dimensions $h_t \times v_t$, it must hold that  $\max(h_t, v_t) \leq n$.
    We will say that such a sequence of choices is \emph{a valid rectangle sequence}. 
    \item \textbf{(Packing Aspect)} Even if a sequence of area choices is valid, it can be the case that it is not possible to use such area choices in a perfect game of~\packit. 
\end{itemize}

This section focuses on studying perfect games through the lens of the first aspect, as it is sometimes enough to determine the \emph{tileability}/\emph{untileability} of grids. 
Despite~\packit~being originally defined for a square grid, from now on we consider $m \times n$ grids as most of our ideas generalize nicely in that setting. Without loss of generality we will assume $n \geq m$ throughout the paper.

In order to state our results, we will need a couple of definitions. We denote by $T_k$ the $k$-{th} \emph{triangle number}, defined as $T_k = \sum_{i=1}^k i = \frac{k(k+1)}{2}$.  Then, for any positive integer $r$, we denote by $\tau(r) = \argmax_{k}\{ T_k \mid T_k \leq r \}$.

An initial observation to understand whether an $m \times n$ grid admits a perfect packing is that the number of rectangles used in perfect ~\packit~games depends entirely on the grid area $m \cdot n$,  and not on its precise width or height

\begin{lemma}
For an $m \times n$ grid there is a unique number $K(m, n)$ such that if the $m \times n$ grid admits a perfect~\packit~game, then such a packing must use exactly $K(m, n)$ rectangles. In particular, $K(m, n) = \tau(m \cdot n).$ 
\label{lemma:K-def}
\end{lemma}

\begin{proof}
    Assume, expecting a contradiction that for some $m \times n$ grid there are two sequences $A := (|A_1|, \ldots, |A_{K_1}|)$ and $A' := (|A'_1|, \ldots, |A'_{K_2}|)$, with $K_1 \neq K_2$, that can be used for perfect packings. Now, note that we must have
    \begin{equation}\label{eq:1}
    \sum_{t=1}^{K_1} |A_t| = m \cdot n =     \sum_{t=1}^{K_2} |A'_t|.
    \end{equation}
    By the game rules, we have that
    \[\sum_{t=1}^{K_1} |A_t| \geq \sum_{t=1}^{K_1} t =  T_{K_1}, \; \quad \text{ and }\quad \; \sum_{t=1}^{K_1} |A_t|\leq  \sum_{t=1}^{K_1} (t+1) = T_{K_1 +1} - 1.\] 
    Using the same analysis for $A'$, and~\Cref{eq:1}, we get
    \[
     \max(T_{K_1}, T_{K_2}) \leq m \cdot n \leq \min(T_{K_1 + 1}, T_{K_2 + 1})- 1. 
    \]
    As $K_1 \neq K_2$, let us assume without loss of generality that $K_1 > K_2$. Using that $T$ is an increasing sequence, we have
    \begin{equation}\label{eq:2}
         T_{K_1} \leq m \cdot n \leq T_{K_2 + 1} - 1. 
    \end{equation}
    Now, as $K_1$ is an integer, $K_1 > K_2$ implies $K_1 \geq K_2 + 1$, from where~\Cref{eq:2} becomes
    \(
          T_{K_1} \leq m \cdot n \leq T_{K_1} - 1, 
    \) a clear contradiction. To obtain the second part of the lemma, note that when $K(m, n) := K_1 = K_2$ we get 
    \[
    T_{K(m, n)} \leq m \cdot n \leq T_{K(m, n) + 1} - 1,
    \]
    from where it follows by the definition of $\tau$ that $K(m, n) = \tau(m \cdot n)$.
\end{proof}

We can now define the notion of \emph{gap}, which intuitively represents the number of turns $t$ in which a rectangle of area $t+1$ must be chosen. 
Let us say that any turn $t$ at which a rectangle of area $t+1$ is chosen is an \emph{expansion turn}. 

\begin{definition}
    For any $m \times n$ grid, we define its \emph{gap}, $\gamma(m, n)$, as
    \[
    \gamma(m, n) = m \cdot n - T_{\tau(m \cdot n)}.
    \]
\end{definition}

\begin{lemma}
    For any sequence of turns that results in a perfect packing of an $m \times n$ grid, the number of expansion turns is exactly $\gamma(m, n)$.
    \label{lemma:gap-def}
\end{lemma}
\begin{proof}
    By~\Cref{lemma:K-def}, there must be exactly $K(m, n) = \tau(m \cdot n)$ turns in such a sequence. If for every turn $t \in \{1, \ldots, \tau(m \cdot n)\}$, a rectangle of area $t$ were to be chosen, then the total area used would be exactly 
    \[
    \sum_{t=1}^{\tau(m \cdot n)} t = T_{\tau(m \cdot n)}.
    \]
    Given that the total area used must be $m \cdot n$, we conclude there must be exactly $m \cdot n - T_{\tau(m \cdot n)}$ expansion turns.
\end{proof}

The next ingredient to analyze whether an $m \times n$ grid admits a perfect packing has to do with prime numbers, as if the area of a rectangle is a prime number $p$, then the only possibles rectangles are $p \times 1$ or $1 \times p$, which can limit our ability to pack it. We define the set $P(m, n)$ as
\[
    P(m, n) = 
    \{ p \mid n < p \leq K(m, n)  \text { and } p \text{ is prime}  \}.
\]

As the next results show, the comparison between the \emph{gap} of a grid and the size of its corresponding $P$ set plays a crucial role in understanding whether or not it allows a perfect packing. In particular,~\Cref{thm:small-gap} shows how small gaps can forbid perfect packings, whereas~\Cref{thm:large-gap} shows how large gaps can also be problematic.


\begin{theorem}[Small gap]
    For any $m \times n$ grid, if $\gamma(m, n) < |P(m, n)|$, then the grid does not allow a perfect game of~\packit.
    \label{thm:small-gap}
\end{theorem}



\begin{figure}[ht]
    \centering
   \begin{tikzpicture}
\onesquare{ 0.0 }{ 4.5 }{\cola}{\numa}
\onesquare{ 0.75 }{ 4.5 }{\colb}{\numb}
\onesquare{ 1.5 }{ 4.5 }{\colb}{\numb}
\onesquare{ 2.25 }{ 4.5 }{\colc}{\numc}
\onesquare{ 3.0 }{ 4.5 }{\colc}{\numc}
\onesquare{ 3.75 }{ 4.5 }{\colc}{\numc}
\onesquare{ 0.0 }{ 3.75 }{\cold}{\numd}
\onesquare{ 0.75 }{ 3.75 }{\cold}{\numd}
\onesquare{ 1.5 }{ 3.75 }{\colg}{\numg}
\onesquare{ 2.25 }{ 3.75 }{\colg}{\numg}
\onesquare{ 3.0 }{ 3.75 }{\colGG}{\;}
\onesquare{ 3.75 }{ 3.75 }{\cole}{\nume}
\onesquare{ 0.0 }{ 3.0 }{\cold}{\numd}
\onesquare{ 0.75 }{ 3.0 }{\cold}{\numd}
\onesquare{ 1.5 }{ 3.0 }{\colg}{\numg}
\onesquare{ 2.25 }{ 3.0 }{\colg}{\numg}
\onesquare{ 3.0 }{ 3.0 }{\colGG}{\;}
\onesquare{ 3.75 }{ 3.0 }{\cole}{\nume}
\onesquare{ 0.0 }{ 2.25 }{\colf}{\numf}
\onesquare{ 0.75 }{ 2.25 }{\colf}{\numf}
\onesquare{ 1.5 }{ 2.25 }{\colg}{\numg}
\onesquare{ 2.25 }{ 2.25 }{\colg}{\numg}
\onesquare{ 3.0 }{ 2.25 }{\colGG}{\;}
\onesquare{ 3.75 }{ 2.25 }{\cole}{\nume}
\onesquare{ 0.0 }{ 1.5 }{\colf}{\numf}
\onesquare{ 0.75 }{ 1.5 }{\colf}{\numf}
\onesquare{ 1.5 }{ 1.5 }{\colg}{\numg}
\onesquare{ 2.25 }{ 1.5 }{\colg}{\numg}
\onesquare{ 3.0 }{ 1.5 }{\colGG}{\;}
\onesquare{ 3.75 }{ 1.5 }{\cole}{\nume}
\onesquare{ 0.0 }{ 0.75 }{\colf}{\numf}
\onesquare{ 0.75 }{ 0.75 }{\colf}{\numf}
\onesquare{ 1.5 }{ 0.75 }{\colGG}{\;}
\onesquare{ 2.25 }{ 0.75 }{\colGG}{\;}
\onesquare{ 3.0 }{ 0.75 }{\colGG}{\;}
\onesquare{ 3.75 }{ 0.75 }{\cole}{\nume}
   \end{tikzpicture}
    \caption{Illustration of the impossibility result for $n=6$ resulting from~\Cref{thm:small-gap}. Even though turns $1$ through $6$ use the minimal possible area, the choice of area $8$ on turn $7$ is enough to make turn $9$ possible, as only $8$ empty cells remain (which is invariant under the concrete choice of packing).}
    \label{fig:small-gap}
\end{figure}

Before a formal proof, let us present some intuition. \Cref{thm:small-gap} considers a gap that is \emph{``too small''}, as the following example shows. Consider $m = n = 6$. One can easily check that, $\tau(6 \cdot 6) = 8$\footnote{A general formula for $\tau(r)$ is not too hard to derive. In particular, $\tau(r) = \left\lfloor\frac{\sqrt{8r + 1}}{2} - \frac{1}{2}\right\rfloor$.}, and therefore the gap results in
\[\gamma(m, n) = m \cdot n - T_{\tau(m \cdot n)} = 6 \cdot 6 - \frac{8 \cdot 9}{2} = 0.\] Then, $K(m, n) = \tau(m \cdot n) = 8$, and thus $P(m, n) = \{7\}$. As $K(m, n) = 8$, any perfect packing of the $6 \times 6$ grid will consist of 8 rectangles. We claim that in turn $7$, the area chosen must be $7$, or in other words, that choosing a rectangle of area $8$ in turn $7$ would forbid a perfect packing. Too see this, consider expecting a contradiction that a rectangle of area $8$ is chosen on turn $7$, and notice that then on the first $8$ turns the smallest sum of areas we can achieve would be \[1 + 2 + 3 + 4 + 5 + 6 + 8 + 8 = 37 > 36,\] a contradiction. On the other hand, given $7$ is a prime number, the only rectangles of area $7$ are a $1 \times 7$ or  a $7 \times 1$ rectangle, neither of which can be packed into a $6 \times 6$ grid. As either area choice for turn $7$ leads to a contradiction, we conclude it is not possible to have a perfect game of~\packit~over the $6 \times 6$ grid.
This example is illustrated in~\Cref{fig:small-gap}, and is generalized in the next proof.
\begin{proof}[Proof of~\Cref{thm:small-gap}]
 
    Let $p \in P(m, n)$. At turn $p$, one must choose between area $p$ or area $p+1$. If area $p$ is chosen, then the rectangle must be either $1 \times p$ or $p \times 1$, due to the primality of $p$. However, by the definition of the set $P(m, n)$ we have $p > n \geq m$, and thus neither the $1 \times p$ nor the $p \times 1$ rectangle can be packed into the $m \times n$ grid. 
    Assume, expecting a contradiction, that $\gamma(m, n) < |P(m,n)|$ and there exists a sequence of turns leading to a perfect packing for the $m \times n$ grid. As a result of the previous argument, every turn $p \in P(m, n)$ must be an expansion turn. As the number of expansion turns is equal to $\gamma(m, n)$ by~\Cref{lemma:gap-def}, we have $\gamma(m, n) \geq |P(m, n)|$, which directly contradicts the assumption.

\end{proof}

\begin{theorem}[Large gap]
    For any $m \times n$ grid, let $\mathbf{1}_{K_p}$ be the indicator variable corresponding to whether
    $K(m, n)+1$ is a prime number or not. Then, the condition
    \[
    \gamma(m, n)   > K(m, n) - |P(m, n)| - \mathbf{1}_{K_p}
    \] 
    implies the $m \times n$ grid does not allow a perfect game of~\packit.
    \label{thm:large-gap}
\end{theorem}

\begin{figure}[t]
    \centering
    \begin{tikzpicture}
\onesquareTwo{ 0.0 }{ 10.799999999999999 }{\cole}{\nume}
\onesquareTwo{ 0.6 }{ 10.799999999999999 }{\cole}{\nume}
\onesquareTwo{ 1.2 }{ 10.799999999999999 }{\cole}{\nume}
\onesquareTwo{ 1.7999999999999998 }{ 10.799999999999999 }{\cole}{\nume}
\onesquareTwo{ 2.4 }{ 10.799999999999999 }{\cole}{\nume}
\onesquareTwo{ 3.0 }{ 10.799999999999999 }{\cole}{\nume}
\onesquareTwo{ 3.5999999999999996 }{ 10.799999999999999 }{\colw}{\numw}
\onesquareTwo{ 4.2 }{ 10.799999999999999 }{\colw}{\numw}
\onesquareTwo{ 4.8 }{ 10.799999999999999 }{\colw}{\numw}
\onesquareTwo{ 5.3999999999999995 }{ 10.799999999999999 }{\colw}{\numw}
\onesquareTwo{ 6.0 }{ 10.799999999999999 }{\colw}{\numw}
\onesquareTwo{ 6.6 }{ 10.799999999999999 }{\colw}{\numw}
\onesquareTwo{ 7.199999999999999 }{ 10.799999999999999 }{\colw}{\numw}
\onesquareTwo{ 7.8 }{ 10.799999999999999 }{\colw}{\numw}
\onesquareTwo{ 8.4 }{ 10.799999999999999 }{\colw}{\numw}
\onesquareTwo{ 9.0 }{ 10.799999999999999 }{\colw}{\numw}
\onesquareTwo{ 9.6 }{ 10.799999999999999 }{\colw}{\numw}
\onesquareTwo{ 10.2 }{ 10.799999999999999 }{\colw}{\numw}
\onesquareTwo{ 0.0 }{ 10.2 }{\coln}{\numn}
\onesquareTwo{ 0.6 }{ 10.2 }{\cold}{\numd}
\onesquareTwo{ 1.2 }{ 10.2 }{\cold}{\numd}
\onesquareTwo{ 1.7999999999999998 }{ 10.2 }{\cold}{\numd}
\onesquareTwo{ 2.4 }{ 10.2 }{\cold}{\numd}
\onesquareTwo{ 3.0 }{ 10.2 }{\cold}{\numd}
\onesquareTwo{ 3.5999999999999996 }{ 10.2 }{\colw}{\numw}
\onesquareTwo{ 4.2 }{ 10.2 }{\colw}{\numw}
\onesquareTwo{ 4.8 }{ 10.2 }{\colw}{\numw}
\onesquareTwo{ 5.3999999999999995 }{ 10.2 }{\colw}{\numw}
\onesquareTwo{ 6.0 }{ 10.2 }{\colw}{\numw}
\onesquareTwo{ 6.6 }{ 10.2 }{\colw}{\numw}
\onesquareTwo{ 7.199999999999999 }{ 10.2 }{\colw}{\numw}
\onesquareTwo{ 7.8 }{ 10.2 }{\colw}{\numw}
\onesquareTwo{ 8.4 }{ 10.2 }{\colw}{\numw}
\onesquareTwo{ 9.0 }{ 10.2 }{\colw}{\numw}
\onesquareTwo{ 9.6 }{ 10.2 }{\colw}{\numw}
\onesquareTwo{ 10.2 }{ 10.2 }{\colw}{\numw}
\onesquareTwo{ 0.0 }{ 9.6 }{\coln}{\numn}
\onesquareTwo{ 0.6 }{ 9.6 }{\colp}{\nump}
\onesquareTwo{ 1.2 }{ 9.6 }{\colp}{\nump}
\onesquareTwo{ 1.7999999999999998 }{ 9.6 }{\colp}{\nump}
\onesquareTwo{ 2.4 }{ 9.6 }{\colp}{\nump}
\onesquareTwo{ 3.0 }{ 9.6 }{\colp}{\nump}
\onesquareTwo{ 3.5999999999999996 }{ 9.6 }{\colp}{\nump}
\onesquareTwo{ 4.2 }{ 9.6 }{\colp}{\nump}
\onesquareTwo{ 4.8 }{ 9.6 }{\colp}{\nump}
\onesquareTwo{ 5.3999999999999995 }{ 9.6 }{\colp}{\nump}
\onesquareTwo{ 6.0 }{ 9.6 }{\colp}{\nump}
\onesquareTwo{ 6.6 }{ 9.6 }{\colp}{\nump}
\onesquareTwo{ 7.199999999999999 }{ 9.6 }{\colp}{\nump}
\onesquareTwo{ 7.8 }{ 9.6 }{\colp}{\nump}
\onesquareTwo{ 8.4 }{ 9.6 }{\colp}{\nump}
\onesquareTwo{ 9.0 }{ 9.6 }{\colp}{\nump}
\onesquareTwo{ 9.6 }{ 9.6 }{\colp}{\nump}
\onesquareTwo{ 10.2 }{ 9.6 }{\colp}{\nump}
\onesquareTwo{ 0.0 }{ 9.0 }{\coln}{\numn}
\onesquareTwo{ 0.6 }{ 9.0 }{\coll}{\numl}
\onesquareTwo{ 1.2 }{ 9.0 }{\colGG}{\;}
\onesquareTwo{ 1.7999999999999998 }{ 9.0 }{\colk}{\numk}
\onesquareTwo{ 2.4 }{ 9.0 }{\colk}{\numk}
\onesquareTwo{ 3.0 }{ 9.0 }{\colk}{\numk}
\onesquareTwo{ 3.5999999999999996 }{ 9.0 }{\colk}{\numk}
\onesquareTwo{ 4.2 }{ 9.0 }{\colk}{\numk}
\onesquareTwo{ 4.8 }{ 9.0 }{\colk}{\numk}
\onesquareTwo{ 5.3999999999999995 }{ 9.0 }{\colk}{\numk}
\onesquareTwo{ 6.0 }{ 9.0 }{\colk}{\numk}
\onesquareTwo{ 6.6 }{ 9.0 }{\colk}{\numk}
\onesquareTwo{ 7.199999999999999 }{ 9.0 }{\colk}{\numk}
\onesquareTwo{ 7.8 }{ 9.0 }{\colk}{\numk}
\onesquareTwo{ 8.4 }{ 9.0 }{\colk}{\numk}
\onesquareTwo{ 9.0 }{ 9.0 }{\colb}{\numb}
\onesquareTwo{ 9.6 }{ 9.0 }{\colb}{\numb}
\onesquareTwo{ 10.2 }{ 9.0 }{\colb}{\numb}
\onesquareTwo{ 0.0 }{ 8.4 }{\coln}{\numn}
\onesquareTwo{ 0.6 }{ 8.4 }{\coll}{\numl}
\onesquareTwo{ 1.2 }{ 8.4 }{\colGG}{\;}
\onesquareTwo{ 1.7999999999999998 }{ 8.4 }{\colg}{\numg}
\onesquareTwo{ 2.4 }{ 8.4 }{\colg}{\numg}
\onesquareTwo{ 3.0 }{ 8.4 }{\colg}{\numg}
\onesquareTwo{ 3.5999999999999996 }{ 8.4 }{\colg}{\numg}
\onesquareTwo{ 4.2 }{ 8.4 }{\colg}{\numg}
\onesquareTwo{ 4.8 }{ 8.4 }{\colg}{\numg}
\onesquareTwo{ 5.3999999999999995 }{ 8.4 }{\colg}{\numg}
\onesquareTwo{ 6.0 }{ 8.4 }{\colg}{\numg}
\onesquareTwo{ 6.6 }{ 8.4 }{\colx}{\numx}
\onesquareTwo{ 7.199999999999999 }{ 8.4 }{\colx}{\numx}
\onesquareTwo{ 7.8 }{ 8.4 }{\colx}{\numx}
\onesquareTwo{ 8.4 }{ 8.4 }{\colx}{\numx}
\onesquareTwo{ 9.0 }{ 8.4 }{\colx}{\numx}
\onesquareTwo{ 9.6 }{ 8.4 }{\cols}{\nums}
\onesquareTwo{ 10.2 }{ 8.4 }{\cols}{\nums}
\onesquareTwo{ 0.0 }{ 7.8 }{\coln}{\numn}
\onesquareTwo{ 0.6 }{ 7.8 }{\coll}{\numl}
\onesquareTwo{ 1.2 }{ 7.8 }{\colu}{\numu}
\onesquareTwo{ 1.7999999999999998 }{ 7.8 }{\colu}{\numu}
\onesquareTwo{ 2.4 }{ 7.8 }{\colj}{\numj}
\onesquareTwo{ 3.0 }{ 7.8 }{\colv}{\numv}
\onesquareTwo{ 3.5999999999999996 }{ 7.8 }{\colv}{\numv}
\onesquareTwo{ 4.2 }{ 7.8 }{\colt}{\numt}
\onesquareTwo{ 4.8 }{ 7.8 }{\colt}{\numt}
\onesquareTwo{ 5.3999999999999995 }{ 7.8 }{\colt}{\numt}
\onesquareTwo{ 6.0 }{ 7.8 }{\colh}{\numh}
\onesquareTwo{ 6.6 }{ 7.8 }{\colx}{\numx}
\onesquareTwo{ 7.199999999999999 }{ 7.8 }{\colx}{\numx}
\onesquareTwo{ 7.8 }{ 7.8 }{\colx}{\numx}
\onesquareTwo{ 8.4 }{ 7.8 }{\colx}{\numx}
\onesquareTwo{ 9.0 }{ 7.8 }{\colx}{\numx}
\onesquareTwo{ 9.6 }{ 7.8 }{\cols}{\nums}
\onesquareTwo{ 10.2 }{ 7.8 }{\cols}{\nums}
\onesquareTwo{ 0.0 }{ 7.199999999999999 }{\coln}{\numn}
\onesquareTwo{ 0.6 }{ 7.199999999999999 }{\coll}{\numl}
\onesquareTwo{ 1.2 }{ 7.199999999999999 }{\colu}{\numu}
\onesquareTwo{ 1.7999999999999998 }{ 7.199999999999999 }{\colu}{\numu}
\onesquareTwo{ 2.4 }{ 7.199999999999999 }{\colj}{\numj}
\onesquareTwo{ 3.0 }{ 7.199999999999999 }{\colv}{\numv}
\onesquareTwo{ 3.5999999999999996 }{ 7.199999999999999 }{\colv}{\numv}
\onesquareTwo{ 4.2 }{ 7.199999999999999 }{\colt}{\numt}
\onesquareTwo{ 4.8 }{ 7.199999999999999 }{\colt}{\numt}
\onesquareTwo{ 5.3999999999999995 }{ 7.199999999999999 }{\colt}{\numt}
\onesquareTwo{ 6.0 }{ 7.199999999999999 }{\colh}{\numh}
\onesquareTwo{ 6.6 }{ 7.199999999999999 }{\colx}{\numx}
\onesquareTwo{ 7.199999999999999 }{ 7.199999999999999 }{\colx}{\numx}
\onesquareTwo{ 7.8 }{ 7.199999999999999 }{\colx}{\numx}
\onesquareTwo{ 8.4 }{ 7.199999999999999 }{\colx}{\numx}
\onesquareTwo{ 9.0 }{ 7.199999999999999 }{\colx}{\numx}
\onesquareTwo{ 9.6 }{ 7.199999999999999 }{\cols}{\nums}
\onesquareTwo{ 10.2 }{ 7.199999999999999 }{\cols}{\nums}
\onesquareTwo{ 0.0 }{ 6.6 }{\coln}{\numn}
\onesquareTwo{ 0.6 }{ 6.6 }{\coll}{\numl}
\onesquareTwo{ 1.2 }{ 6.6 }{\colu}{\numu}
\onesquareTwo{ 1.7999999999999998 }{ 6.6 }{\colu}{\numu}
\onesquareTwo{ 2.4 }{ 6.6 }{\colj}{\numj}
\onesquareTwo{ 3.0 }{ 6.6 }{\colv}{\numv}
\onesquareTwo{ 3.5999999999999996 }{ 6.6 }{\colv}{\numv}
\onesquareTwo{ 4.2 }{ 6.6 }{\colt}{\numt}
\onesquareTwo{ 4.8 }{ 6.6 }{\colt}{\numt}
\onesquareTwo{ 5.3999999999999995 }{ 6.6 }{\colt}{\numt}
\onesquareTwo{ 6.0 }{ 6.6 }{\colh}{\numh}
\onesquareTwo{ 6.6 }{ 6.6 }{\colx}{\numx}
\onesquareTwo{ 7.199999999999999 }{ 6.6 }{\colx}{\numx}
\onesquareTwo{ 7.8 }{ 6.6 }{\colx}{\numx}
\onesquareTwo{ 8.4 }{ 6.6 }{\colx}{\numx}
\onesquareTwo{ 9.0 }{ 6.6 }{\colx}{\numx}
\onesquareTwo{ 9.6 }{ 6.6 }{\cols}{\nums}
\onesquareTwo{ 10.2 }{ 6.6 }{\cols}{\nums}
\onesquareTwo{ 0.0 }{ 6.0 }{\coln}{\numn}
\onesquareTwo{ 0.6 }{ 6.0 }{\coll}{\numl}
\onesquareTwo{ 1.2 }{ 6.0 }{\colu}{\numu}
\onesquareTwo{ 1.7999999999999998 }{ 6.0 }{\colu}{\numu}
\onesquareTwo{ 2.4 }{ 6.0 }{\colj}{\numj}
\onesquareTwo{ 3.0 }{ 6.0 }{\colv}{\numv}
\onesquareTwo{ 3.5999999999999996 }{ 6.0 }{\colv}{\numv}
\onesquareTwo{ 4.2 }{ 6.0 }{\colt}{\numt}
\onesquareTwo{ 4.8 }{ 6.0 }{\colt}{\numt}
\onesquareTwo{ 5.3999999999999995 }{ 6.0 }{\colt}{\numt}
\onesquareTwo{ 6.0 }{ 6.0 }{\colh}{\numh}
\onesquareTwo{ 6.6 }{ 6.0 }{\colx}{\numx}
\onesquareTwo{ 7.199999999999999 }{ 6.0 }{\colx}{\numx}
\onesquareTwo{ 7.8 }{ 6.0 }{\colx}{\numx}
\onesquareTwo{ 8.4 }{ 6.0 }{\colx}{\numx}
\onesquareTwo{ 9.0 }{ 6.0 }{\colx}{\numx}
\onesquareTwo{ 9.6 }{ 6.0 }{\cols}{\nums}
\onesquareTwo{ 10.2 }{ 6.0 }{\cols}{\nums}
\onesquareTwo{ 0.0 }{ 5.3999999999999995 }{\coln}{\numn}
\onesquareTwo{ 0.6 }{ 5.3999999999999995 }{\coll}{\numl}
\onesquareTwo{ 1.2 }{ 5.3999999999999995 }{\colu}{\numu}
\onesquareTwo{ 1.7999999999999998 }{ 5.3999999999999995 }{\colu}{\numu}
\onesquareTwo{ 2.4 }{ 5.3999999999999995 }{\colj}{\numj}
\onesquareTwo{ 3.0 }{ 5.3999999999999995 }{\colv}{\numv}
\onesquareTwo{ 3.5999999999999996 }{ 5.3999999999999995 }{\colv}{\numv}
\onesquareTwo{ 4.2 }{ 5.3999999999999995 }{\colt}{\numt}
\onesquareTwo{ 4.8 }{ 5.3999999999999995 }{\colt}{\numt}
\onesquareTwo{ 5.3999999999999995 }{ 5.3999999999999995 }{\colt}{\numt}
\onesquareTwo{ 6.0 }{ 5.3999999999999995 }{\colh}{\numh}
\onesquareTwo{ 6.6 }{ 5.3999999999999995 }{\colf}{\numf}
\onesquareTwo{ 7.199999999999999 }{ 5.3999999999999995 }{\colq}{\numq}
\onesquareTwo{ 7.8 }{ 5.3999999999999995 }{\colq}{\numq}
\onesquareTwo{ 8.4 }{ 5.3999999999999995 }{\coli}{\numi}
\onesquareTwo{ 9.0 }{ 5.3999999999999995 }{\coli}{\numi}
\onesquareTwo{ 9.6 }{ 5.3999999999999995 }{\cols}{\nums}
\onesquareTwo{ 10.2 }{ 5.3999999999999995 }{\cols}{\nums}
\onesquareTwo{ 0.0 }{ 4.8 }{\coln}{\numn}
\onesquareTwo{ 0.6 }{ 4.8 }{\coll}{\numl}
\onesquareTwo{ 1.2 }{ 4.8 }{\colu}{\numu}
\onesquareTwo{ 1.7999999999999998 }{ 4.8 }{\colu}{\numu}
\onesquareTwo{ 2.4 }{ 4.8 }{\colj}{\numj}
\onesquareTwo{ 3.0 }{ 4.8 }{\colv}{\numv}
\onesquareTwo{ 3.5999999999999996 }{ 4.8 }{\colv}{\numv}
\onesquareTwo{ 4.2 }{ 4.8 }{\colt}{\numt}
\onesquareTwo{ 4.8 }{ 4.8 }{\colt}{\numt}
\onesquareTwo{ 5.3999999999999995 }{ 4.8 }{\colt}{\numt}
\onesquareTwo{ 6.0 }{ 4.8 }{\colh}{\numh}
\onesquareTwo{ 6.6 }{ 4.8 }{\colf}{\numf}
\onesquareTwo{ 7.199999999999999 }{ 4.8 }{\colq}{\numq}
\onesquareTwo{ 7.8 }{ 4.8 }{\colq}{\numq}
\onesquareTwo{ 8.4 }{ 4.8 }{\coli}{\numi}
\onesquareTwo{ 9.0 }{ 4.8 }{\coli}{\numi}
\onesquareTwo{ 9.6 }{ 4.8 }{\cols}{\nums}
\onesquareTwo{ 10.2 }{ 4.8 }{\cols}{\nums}
\onesquareTwo{ 0.0 }{ 4.2 }{\coln}{\numn}
\onesquareTwo{ 0.6 }{ 4.2 }{\coll}{\numl}
\onesquareTwo{ 1.2 }{ 4.2 }{\colu}{\numu}
\onesquareTwo{ 1.7999999999999998 }{ 4.2 }{\colu}{\numu}
\onesquareTwo{ 2.4 }{ 4.2 }{\colj}{\numj}
\onesquareTwo{ 3.0 }{ 4.2 }{\colv}{\numv}
\onesquareTwo{ 3.5999999999999996 }{ 4.2 }{\colv}{\numv}
\onesquareTwo{ 4.2 }{ 4.2 }{\colt}{\numt}
\onesquareTwo{ 4.8 }{ 4.2 }{\colt}{\numt}
\onesquareTwo{ 5.3999999999999995 }{ 4.2 }{\colt}{\numt}
\onesquareTwo{ 6.0 }{ 4.2 }{\colh}{\numh}
\onesquareTwo{ 6.6 }{ 4.2 }{\colf}{\numf}
\onesquareTwo{ 7.199999999999999 }{ 4.2 }{\colq}{\numq}
\onesquareTwo{ 7.8 }{ 4.2 }{\colq}{\numq}
\onesquareTwo{ 8.4 }{ 4.2 }{\coli}{\numi}
\onesquareTwo{ 9.0 }{ 4.2 }{\coli}{\numi}
\onesquareTwo{ 9.6 }{ 4.2 }{\cols}{\nums}
\onesquareTwo{ 10.2 }{ 4.2 }{\cols}{\nums}
\onesquareTwo{ 0.0 }{ 3.5999999999999996 }{\coln}{\numn}
\onesquareTwo{ 0.6 }{ 3.5999999999999996 }{\coll}{\numl}
\onesquareTwo{ 1.2 }{ 3.5999999999999996 }{\colu}{\numu}
\onesquareTwo{ 1.7999999999999998 }{ 3.5999999999999996 }{\colu}{\numu}
\onesquareTwo{ 2.4 }{ 3.5999999999999996 }{\colj}{\numj}
\onesquareTwo{ 3.0 }{ 3.5999999999999996 }{\colv}{\numv}
\onesquareTwo{ 3.5999999999999996 }{ 3.5999999999999996 }{\colv}{\numv}
\onesquareTwo{ 4.2 }{ 3.5999999999999996 }{\colr}{\numr}
\onesquareTwo{ 4.8 }{ 3.5999999999999996 }{\colr}{\numr}
\onesquareTwo{ 5.3999999999999995 }{ 3.5999999999999996 }{\colr}{\numr}
\onesquareTwo{ 6.0 }{ 3.5999999999999996 }{\colh}{\numh}
\onesquareTwo{ 6.6 }{ 3.5999999999999996 }{\colf}{\numf}
\onesquareTwo{ 7.199999999999999 }{ 3.5999999999999996 }{\colq}{\numq}
\onesquareTwo{ 7.8 }{ 3.5999999999999996 }{\colq}{\numq}
\onesquareTwo{ 8.4 }{ 3.5999999999999996 }{\coli}{\numi}
\onesquareTwo{ 9.0 }{ 3.5999999999999996 }{\coli}{\numi}
\onesquareTwo{ 9.6 }{ 3.5999999999999996 }{\cols}{\nums}
\onesquareTwo{ 10.2 }{ 3.5999999999999996 }{\cols}{\nums}
\onesquareTwo{ 0.0 }{ 3.0 }{\coln}{\numn}
\onesquareTwo{ 0.6 }{ 3.0 }{\coll}{\numl}
\onesquareTwo{ 1.2 }{ 3.0 }{\colu}{\numu}
\onesquareTwo{ 1.7999999999999998 }{ 3.0 }{\colu}{\numu}
\onesquareTwo{ 2.4 }{ 3.0 }{\colj}{\numj}
\onesquareTwo{ 3.0 }{ 3.0 }{\colv}{\numv}
\onesquareTwo{ 3.5999999999999996 }{ 3.0 }{\colv}{\numv}
\onesquareTwo{ 4.2 }{ 3.0 }{\colr}{\numr}
\onesquareTwo{ 4.8 }{ 3.0 }{\colr}{\numr}
\onesquareTwo{ 5.3999999999999995 }{ 3.0 }{\colr}{\numr}
\onesquareTwo{ 6.0 }{ 3.0 }{\colh}{\numh}
\onesquareTwo{ 6.6 }{ 3.0 }{\colf}{\numf}
\onesquareTwo{ 7.199999999999999 }{ 3.0 }{\colq}{\numq}
\onesquareTwo{ 7.8 }{ 3.0 }{\colq}{\numq}
\onesquareTwo{ 8.4 }{ 3.0 }{\coli}{\numi}
\onesquareTwo{ 9.0 }{ 3.0 }{\coli}{\numi}
\onesquareTwo{ 9.6 }{ 3.0 }{\cols}{\nums}
\onesquareTwo{ 10.2 }{ 3.0 }{\cols}{\nums}
\onesquareTwo{ 0.0 }{ 2.4 }{\coln}{\numn}
\onesquareTwo{ 0.6 }{ 2.4 }{\coll}{\numl}
\onesquareTwo{ 1.2 }{ 2.4 }{\colu}{\numu}
\onesquareTwo{ 1.7999999999999998 }{ 2.4 }{\colu}{\numu}
\onesquareTwo{ 2.4 }{ 2.4 }{\colj}{\numj}
\onesquareTwo{ 3.0 }{ 2.4 }{\colv}{\numv}
\onesquareTwo{ 3.5999999999999996 }{ 2.4 }{\colv}{\numv}
\onesquareTwo{ 4.2 }{ 2.4 }{\colr}{\numr}
\onesquareTwo{ 4.8 }{ 2.4 }{\colr}{\numr}
\onesquareTwo{ 5.3999999999999995 }{ 2.4 }{\colr}{\numr}
\onesquareTwo{ 6.0 }{ 2.4 }{\cola}{\numa}
\onesquareTwo{ 6.6 }{ 2.4 }{\colf}{\numf}
\onesquareTwo{ 7.199999999999999 }{ 2.4 }{\colq}{\numq}
\onesquareTwo{ 7.8 }{ 2.4 }{\colq}{\numq}
\onesquareTwo{ 8.4 }{ 2.4 }{\colo}{\numo}
\onesquareTwo{ 9.0 }{ 2.4 }{\colo}{\numo}
\onesquareTwo{ 9.6 }{ 2.4 }{\colo}{\numo}
\onesquareTwo{ 10.2 }{ 2.4 }{\colo}{\numo}
\onesquareTwo{ 0.0 }{ 1.7999999999999998 }{\coln}{\numn}
\onesquareTwo{ 0.6 }{ 1.7999999999999998 }{\coll}{\numl}
\onesquareTwo{ 1.2 }{ 1.7999999999999998 }{\colu}{\numu}
\onesquareTwo{ 1.7999999999999998 }{ 1.7999999999999998 }{\colu}{\numu}
\onesquareTwo{ 2.4 }{ 1.7999999999999998 }{\colj}{\numj}
\onesquareTwo{ 3.0 }{ 1.7999999999999998 }{\colv}{\numv}
\onesquareTwo{ 3.5999999999999996 }{ 1.7999999999999998 }{\colv}{\numv}
\onesquareTwo{ 4.2 }{ 1.7999999999999998 }{\colr}{\numr}
\onesquareTwo{ 4.8 }{ 1.7999999999999998 }{\colr}{\numr}
\onesquareTwo{ 5.3999999999999995 }{ 1.7999999999999998 }{\colr}{\numr}
\onesquareTwo{ 6.0 }{ 1.7999999999999998 }{\cola}{\numa}
\onesquareTwo{ 6.6 }{ 1.7999999999999998 }{\colf}{\numf}
\onesquareTwo{ 7.199999999999999 }{ 1.7999999999999998 }{\colq}{\numq}
\onesquareTwo{ 7.8 }{ 1.7999999999999998 }{\colq}{\numq}
\onesquareTwo{ 8.4 }{ 1.7999999999999998 }{\colo}{\numo}
\onesquareTwo{ 9.0 }{ 1.7999999999999998 }{\colo}{\numo}
\onesquareTwo{ 9.6 }{ 1.7999999999999998 }{\colo}{\numo}
\onesquareTwo{ 10.2 }{ 1.7999999999999998 }{\colo}{\numo}
\onesquareTwo{ 0.0 }{ 1.2 }{\colm}{\numm}
\onesquareTwo{ 0.6 }{ 1.2 }{\colm}{\numm}
\onesquareTwo{ 1.2 }{ 1.2 }{\colm}{\numm}
\onesquareTwo{ 1.7999999999999998 }{ 1.2 }{\colm}{\numm}
\onesquareTwo{ 2.4 }{ 1.2 }{\colm}{\numm}
\onesquareTwo{ 3.0 }{ 1.2 }{\colm}{\numm}
\onesquareTwo{ 3.5999999999999996 }{ 1.2 }{\colm}{\numm}
\onesquareTwo{ 4.2 }{ 1.2 }{\colr}{\numr}
\onesquareTwo{ 4.8 }{ 1.2 }{\colr}{\numr}
\onesquareTwo{ 5.3999999999999995 }{ 1.2 }{\colr}{\numr}
\onesquareTwo{ 6.0 }{ 1.2 }{\colc}{\numc}
\onesquareTwo{ 6.6 }{ 1.2 }{\colc}{\numc}
\onesquareTwo{ 7.199999999999999 }{ 1.2 }{\colq}{\numq}
\onesquareTwo{ 7.8 }{ 1.2 }{\colq}{\numq}
\onesquareTwo{ 8.4 }{ 1.2 }{\colo}{\numo}
\onesquareTwo{ 9.0 }{ 1.2 }{\colo}{\numo}
\onesquareTwo{ 9.6 }{ 1.2 }{\colo}{\numo}
\onesquareTwo{ 10.2 }{ 1.2 }{\colo}{\numo}
\onesquareTwo{ 0.0 }{ 0.6 }{\colm}{\numm}
\onesquareTwo{ 0.6 }{ 0.6 }{\colm}{\numm}
\onesquareTwo{ 1.2 }{ 0.6 }{\colm}{\numm}
\onesquareTwo{ 1.7999999999999998 }{ 0.6 }{\colm}{\numm}
\onesquareTwo{ 2.4 }{ 0.6 }{\colm}{\numm}
\onesquareTwo{ 3.0 }{ 0.6 }{\colm}{\numm}
\onesquareTwo{ 3.5999999999999996 }{ 0.6 }{\colm}{\numm}
\onesquareTwo{ 4.2 }{ 0.6 }{\colr}{\numr}
\onesquareTwo{ 4.8 }{ 0.6 }{\colr}{\numr}
\onesquareTwo{ 5.3999999999999995 }{ 0.6 }{\colr}{\numr}
\onesquareTwo{ 6.0 }{ 0.6 }{\colc}{\numc}
\onesquareTwo{ 6.6 }{ 0.6 }{\colc}{\numc}
\onesquareTwo{ 7.199999999999999 }{ 0.6 }{\colq}{\numq}
\onesquareTwo{ 7.8 }{ 0.6 }{\colq}{\numq}
\onesquareTwo{ 8.4 }{ 0.6 }{\colo}{\numo}
\onesquareTwo{ 9.0 }{ 0.6 }{\colo}{\numo}
\onesquareTwo{ 9.6 }{ 0.6 }{\colo}{\numo}
\onesquareTwo{ 10.2 }{ 0.6 }{\colo}{\numo}

    \end{tikzpicture}
    \caption{Illustration of the impossibility result for $n=18$ (\Cref{thm:large-gap}). Even though almost each rectangle $t$ has area $t+1$, except for $t \in \{18, 22\}$ (where $t+1 > n$ is prime), the total area covered by turn $24$ is only $322 = 18^2- 2$, and naturally it is not possible to fill in the two remaining cells in turn $25$.  }
    \label{fig:large-gap}
\end{figure}

Before the proof, let us present some intuition for \Cref{thm:large-gap}.
Consider $m = n = 18$ (this example is illustrated in~\Cref{fig:large-gap}). As a result, $\tau(18 \cdot 18) = 24$
, and therefore the gap is
\[\gamma(m, n) = m \cdot n - T_{\tau(m \cdot n)} = 18 \cdot 18 - \frac{24 \cdot 25}{2} = 24.\] 
We also have $K(m, n) = \tau(m \cdot n) = 24$, implying that any perfect packing of the $18 \times 18$ grid will consist of $K(m, n) = 24$ rectangles. We claim that on turn 18, both choices of area, 18 and 19, lead to contradictions. Let us see what happens if area 18 is chosen on turn 18. In this case, even if area $t+1$ is chosen on every turn $t \neq 18$, the maximum sum of the areas we can achieve is \[2 + 3 + \ldots + 17 + 18 + \mathbf{18} + 20 + \ldots + 25 = 323 < 324,\] implying the $324$ cells of the $18 \times 18$ grid cannot be covered. 
On the other hand, if area $19$ is chosen on turn $18$, we run into a different issue: as 19 is a prime number it only allows for the rectangles $1 \times 19$ or $19 \times 1$, neither of which be can be packed into the $18 \times 18$ grid. As both cases lead to an impossibility, we conclude it is not possible to have a perfect game of~\packit~over the $18 \times 18$ grid. The proof for \Cref{thm:large-gap} generalizes this example.


\begin{proof}[Proof of~\Cref{thm:large-gap}]
    Let $p \in P(m, n)$. As $p \leq K(m, n)$ by definition of $P(m, n)$, turn $p-1$ is necessarily part of any perfect packing. At turn $p - 1$, one must choose between area $p - 1$ or area $p$. If area $p$ is chosen, then the rectangle must be either $1 \times p$ or $p \times 1$, due to the primality of $p$. However, by the definition of the set $P(m, n)$ we have $p > n \geq m$, and thus neither the $1 \times p$ nor the $p \times 1$ rectangle can be packed into the $m \times n$ grid. We conclude that for each $p \in P(m, n)$, the turn $p-1$ is not an expansion turn.

    If $K(m, n) + 1$ is prime, then the rectangle turn $K(m, n)$ cannot be an expansion turn. By definition, $K(m, n) + 1 \not\in P(m, n)$, so the number of turns that are not expansion turns is at least  $|P(m, n)| + \mathbf{1}_{K_p}$.  By \Cref{lemma:K-def}, the number of expansion turns is exactly $\gamma(m, n)$, which together with the previous fact implies that the total number of turns is at least 
    \begin{equation}
    |P(m, n)| + \mathbf{1}_{K_p} + \gamma(m, n).
    \label{eq:lb-turns}
    \end{equation}
    Suppose, expecting a contradiction that 
    \begin{equation}
    \gamma(m, n) > K(m, n) - |P(m, n)| - \mathbf{1}_{K_p},\label{eq:assumpt}
    \end{equation} and yet there exists a sequence of turns leading to a perfect packing for the $m \times n$ grid. By combining~\Cref{eq:lb-turns} and~\Cref{eq:assumpt}, the total number of turns is at least 
    \begin{align*}
    |P(m, n)| + \mathbf{1}_{K_p} + \gamma(m, n) &>  |P(m, n)| + \mathbf{1}_{K_p}  +  K(m, n) - |P(m, n)| - \mathbf{1}_{K_p}\\
    &= K(m, n),
    \end{align*}
    which is a contradiction, given the total number of turns must be exactly $K(m, n)$ according to~\Cref{lemma:K-def}.


\end{proof}

Combining~\Cref{thm:small-gap} and \Cref{thm:large-gap}, we obtain a range  of values for the gap of an $m \times n$ grid in which perfect packings are \emph{a priori} possible. So far, we have not found any examples of $m \times n$ grids whose gap belongs in this range and yet no perfect packings exist. Therefore, we pose the following conjecture

\begin{conjecture}
    Let $m \leq n$ be positive integers. Then, if 
    \[
       |P(m, n)| \, \leq  \, \gamma(m, n) \, \leq \, K(m, n) - |P(m, n)| - \mathbf{1}_{K_p}, 
    \]
    it is possible to complete a perfect game of~\packit~for the $m \times n$ grid.
    \label{conjecture:main}
\end{conjecture}




Interestingly,~\Cref{thm:small-gap} is enough to construct infinite families of $n\times n$ grids that do not admit perfect packings.

\begin{theorem}
    There are infinitely many positive integers $n$ such that the $n \times n$ grid does not admit a perfect game of~\packit.
    \label{thm:basic-imposibility}
\end{theorem}
\begin{proof}
    By~\Cref{thm:small-gap}, if suffices to show that there are infinitely many values of $n$ such that $\gamma(n, n) = 1$ and $|P(n, n)| > 1$. First, consider the following claim.
    \begin{claim}
        For every $n \geq 100$, we have $K(n, n) \geq 1.4n$.
        \label{claim:1.4}
    \end{claim}
    \begin{proof}[Proof of~\Cref{claim:1.4}]
        Let $\ell = \lfloor 1.4n \rfloor$. It suffices to argue that $T_\ell \leq n^2$. As $\ell > n \geq 100$, we have $\ell < \frac{1}{100}\ell^2$, which we can use as follows.
        \begin{align*}
            T_\ell = \frac{\ell^2 + \ell}{2} \leq \frac{\frac{101\ell^2}{100}}{2} = 101\ell^2/200,
        \end{align*}
        and conclude by noting that 
        \[
        101\ell^2/200 \leq \frac{101}{200} \cdot \left( \frac{140}{100}n\right)^2 = \frac{1\,979\,600}{2\,000\,000}n^2 \leq n^2.
        \]
    \end{proof}
    Now, Schoenfeld proved in~\cite{schoenfeldSharperBoundsChebyshev1976} that for every $n > 3 \cdot 10^6$, there is always a prime number between $n$ and $\left(1 + \frac{1}{16957}\right)n$, which applied twice gives us that there are always (at least) two prime numbers between $n$ and $\left(1 + \frac{1}{16957}\right)^2 n \leq 1.4n$. Therefore, for $n > 3 \cdot 10^6$ we always have $|P(n, n)| > 1$. 
    It remains to prove that $\gamma(n, n) = 1$ infinitely often. We do this by using the theory of generalized Pell's equation. Indeed, the condition $\gamma(n, n) = 1$ can be written, by using notation $K := K(n, n)$, as 
    \begin{equation}
    n^2 - \frac{K(K+1)}{2} = 1,
    \label{eq:original}
     \end{equation}
    which after multiplying both sides by $8$ and rearranging is equivalent to
    \[
    8n^2 - (2K + 1)^2 = 7.
    \]
   Introducing the variable $t \coloneqq (2K+1)$ we consider the following equations.
    \begin{align}
        t^2 - 8n^2 &= -7, \label{eq:non-homogeneous}\\
        \left(t^{(h)}\right)^2 - 8\left(n^{(h)}\right)^2 &= 1.  \label{eq:homogeneous}
    \end{align}
    While~\Cref{eq:homogeneous} presents an \emph{``homogeneous''} Pell equation, for which it is well known that infinitely many solutions exist over the positive integers (cf. the problem of square triangular numbers~\cite{Barbeau2003-tc}), ~\Cref{eq:non-homogeneous} corresponds to a \emph{``non-homogeneous''} equation, less frequently studied. Similarly to the theory of ordinary differential equations, we can obtain a  set of solutions to the non-homogeneous equation by combining one \emph{initial solution} for it with a set of solutions to its homogeneous counterpart.
    Indeed, assume the existence of a solution $(n_0, t_0)$ to~\Cref{eq:non-homogeneous} over the positive integers, and $\left(n^{(h)}_i, t^{(h)}_i\right)$ a sequence of solutions to~\Cref{eq:homogeneous} over the positive integers, whose existence is standard (see e.g.,~\cite{Barbeau2003-tc}). 

    
    \begin{claim}
        The sequence $(n_i, t_i)$, defined as
        \begin{equation}
      (n_i, t_i) \coloneqq  \left(t_0 t^{(h)}_i + 8 n_0n^{(h)}_i, \; t_0 n^{(h)}_i + n_0 t^{(h)}_i\right),
      \label{eq:new_solutions}
      \end{equation}
      is an infinite family of solutions
     of~\Cref{eq:non-homogeneous} over the positive integers.
     \label{claim:infinite-solutions}
    \end{claim}
    \begin{proof}[Proof of~\Cref{claim:infinite-solutions}]
    \newcommand{\sn}{n^{(h)}_i}
    \newcommand{\tn}{t^{(h)}_i}
        By assumption, $(n_0, t_0)$ is a solution of~\Cref{eq:non-homogeneous}, and $\left(\sn, \tn\right)$ is a solution of~\Cref{eq:homogeneous}. Thus, we have
        \begin{align*}
           -7 &= \left(t_0^2 - 8n_0^2 \right)\left(\left(\tn\right)^2 - 8\left(\sn\right)^2\right) \\
           &= (t_0 + \sqrt{8}n_0)(t_0 - \sqrt{8}n_0)\left(\tn + \sqrt{8}\sn\right)\left(\tn - \sqrt{8}\sn \right)\\
           &= \left[(t_0 + \sqrt{8}n_0)  \left(\tn + \sqrt{8}\sn\right)\right] \cdot \left[ (t_0 - \sqrt{8}n_0)\left(\tn - \sqrt{8}\sn \right) \right]\\
           &=\left[\left(t_0\tn + 8n_0 \sn\right) + \sqrt{8}\left(t_0\sn + n_0\tn\right) \right]\\ &\,\,\; \cdot\left[\left(t_0\tn + 8n_0 \sn\right) - \sqrt{8}\left(t_0\sn +n_0\tn\right)\right]\\
           &= \left(t_0\tn + 8n_0 \sn\right)^2 - 8\left(t_0\sn + n_0\tn\right)^2\\
           &= n_i^2 - 8t_i^2. \qedhere
        \end{align*}
    \end{proof}
    
   As we can provide an initial solution~$(n_0, t_0) := (11, 31)$ to~\Cref{eq:non-homogeneous}, we conclude by~\Cref{claim:infinite-solutions} that it has infinitely many solutions over the positive integers. 
We now finish the proof by the following claim.

   \begin{claim}
       Every solution $(n_i, t_i)$ to~\Cref{eq:non-homogeneous} over the positive integers with $n_i > 3 \cdot 10^6$ corresponds to a value of $n$ such that the $n\times n$ grid does not admit a perfect game of~\packit.
       \label{claim:back-to-original}
   \end{claim}
   \begin{proof}[Proof of~\Cref{claim:back-to-original}]
       Let $(n_i, t_i)$ be a solution to~\Cref{eq:non-homogeneous} and let us argue that the $n_i \times n_i$ does not admit a perfect game of~\packit. First, consider that $t_i$ must be odd, as $t_i^2 = 1 + 8n_i^2$, by~\Cref{eq:non-homogeneous}. Therefore $(t_i - 1)/2$ is a positive integer. We now a argue that $(t_i -1)/2$ indeed matches the definition of $K(n_i, n_i)$.  Let us denote $(t_i-1)/2$ by $K'$, and we will argue that indeed $K' = K(n_i, n_i)$. To see, this, consider that as~\Cref{eq:non-homogeneous} has the same set of solutions as~\Cref{eq:original}, it must be the case that 
       \[
       n_i^2 - \frac{K' (K'+1)}{2} = 1,
       \]
       implying that $T_{K'} = n_i^2 - 1 \leq n_i^2$. Moreover, we have that \[ T_{K'+1} = T_{K'} + (K'+1) = n_i^2 + K' > n_i^2, \]
       thereby confirming that $K' = \tau(n_i^2) = K(n_i, n_i)$. Taking $n := n_i$, we have by construction that
       \(
            \gamma(n, n) = 1,
       \)
       and as $n > 3 \cdot 10^6$ we have $|P(n, n)| > 1$. Therefore the condition of~\Cref{thm:small-gap} applies to $n$, implying the $n \times n$ grid does not admit a perfect packing. This concludes the proof of the entire theorem.
   \end{proof}
    \renewcommand{\qedsymbol}{}
\end{proof}

Let us define notation $\gamma^{-1}(c)$ to denote the set 
\(
    \{ n \in \mathbb{N}^{>0} \; | \; \gamma(n, n) = c\}.
\)
The previous proof showed that there are infinitely many values of $n \in \gamma^{-1}(1)$ that do not admit perfect packings. We now show a much stronger statement.

\begin{theorem}
For every value $c \geq 0$, only a finite number of values $n \in \gamma^{-1}(c)$ allow for a perfect packing of the $n \times n$ grid.

\end{theorem}

\begin{proof}
By~\Cref{thm:small-gap}, it suffices to show that for every value $c \geq 0$, there are only finitely many values of $n$ such that 
\[
    |P(n, n)| = \{ n < p \leq K(n, n) \mid p \text { is prime}  \} \leq c
\]
We will do so by using the following improvement on Bertrand's postulate due to Dusart.

\begin{proposition}[\cite{dusar1998}]
    For every value of $n > 3275$, there exists a prime number $p$ such that 
    \[
    n < p \leq n\left(1 + \frac{1}{2\ln^2 n} \right).
    \]
    \label{prop:dusar}
\end{proposition}
In particular, if we apply~\Cref{prop:dusar} exactly $c+1$ times, we obtain that 

\[
\left|\left\{ n < p \leq n \left(1 + \frac{1}{2\ln^2 n} \right)^{c+1} \;\Big\vert\;\; p \text { is prime}  \right\}\right| \geq c+1, \quad \text{for every } n > 3275.
\]
 Now, let us see that for every sufficiently large $n$ it holds that 
\[
 n\left(1 + \frac{1}{2\ln^2 n} \right)^{c+1} \leq K(n, n),
\]
which will be enough to conclude. Indeed, recall that by~\Cref{claim:1.4} we have that $K(n, n) \geq 1.4n$ for $n \geq 100$, and hence it only remains for us to show that for sufficiently large $n$ we have 
\[
    \left(1 + \frac{1}{2\ln^2 n} \right)^{c+1} \leq 1.4,
\]
which must be true since the LHS is monotonically decreasing in $n$ and its limit when $n$ goes to infinity is $1$.
\end{proof}





\begin{theorem}
    For every even $n\geq 2$, the $2 \times \frac{n^2}{2}$ grid always admits a perfect game of~\packit.
    \label{thm:2xn2}
\end{theorem}

\begin{proof}

The proof is constructive. Let $K \coloneqq K\left(2, \frac{n^2}{2}\right)$. As a first step, we place the first $n-1$ rectangles (i.e., $1 \times t$ for $t \in \{1, \ldots, n-1\}$) in the first row, one after another, thus covering the first $\frac{n(n-1)}{2} < \frac{n^2}{2}$ cells of the first row. 
Some of these rectangles will be \emph{expanded} later on in order to fill up the first row, meaning that the rectangle $1 \times t$, used in turn $t$ will be replaced by a rectangle $1 \times (t+1)$. The remaining $K-n-1$ rectangles, for $t \geq n$, will be placed on the second row. We might have to move some rectangles from the first row to the second row or vice-versa. The proof proceeds by cases over $\gamma\left(2, \frac{n^2}{2}\right)$, which we will abbreviate by $\gamma_n$ to alleviate notation.\\


\subparagraph*{(\textbf{Case 1:} $\gamma_n \leq \frac{n}{2}$)}
As introduced earlier, the  first step is to place the first $n - 1$ rectangles in the row one and the rest in row two. For the moment we do not care if row two is too long or row one too short; we will deal with that in a moment. Next, expand the first $\gamma_n$ rectangles of row one. Originally, row one was $\frac{n^2}{2} - \frac{n(n-1)}{2} = \frac{n}{2}$ cells too short, and after the expansion of the first $\gamma_n$ rectangles it is $\frac{n}{2} - \gamma_n$ cells too short. By~\Cref{lemma:gap-def}, the $\gamma_n$ expansions in row one, guarantee that the total area of rectangles in row one and two adds up to exactly $n^2$. As a result, row two must be exactly $\frac{n}{2} - \gamma_n$ cells too large.  If $\gamma_n$ were to be exactly $\frac{n}{2}$, we would be done immediately.
Otherwise, we will swap a rectangle from row one with a rectangle from row two. Indeed, note that $r_{\frac{n}{2} + \gamma_n}$, the $1 \times \frac{n}{2} + \gamma_n$ rectangle, is still on row one, and it was not expanded. Therefore, we can swap $r_{\frac{n}{2} + \gamma_n}$ (from row one) with $r_n$ (from row two). As a result, row one has grown by $n - \left(\frac{n}{2} + \gamma_n \right) = \frac{n}{2} - \gamma_n$ cells, and row two has shrunk by the same amount. Therefore both rows have reached their desired length. This case is illustrated in~\Cref{fig:case-1-proof}.\\
\begin{figure}
    \centering
    \begin{tikzpicture}
        \onesquareTwo{ 0.0 }{ 0.0 }{\cola}{\numa}
        \onesquareTwo{ 0.6 }{ 0.0 }{\colb}{\numb}
            \onesquareTwo{ 1.2 }{ 0.0 }{\colb}{\numb}
        \onesquareTwo{ 1.8 }{ 0.0 }{\colc}{\numc}
            \onesquareTwo{ 2.4 }{ 0.0 }{\colc}{\numc}
            \onesquareTwo{ 3.0 }{ 0.0 }{\colc}{\numc}
        \onesquareTwo{ 3.6 }{ 0.0 }{\colGG}{\;}
            \onesquareTwo{ 4.2 }{ 0.0 }{\colGG}{\;}
        \onesquareTwo{ 0.0 }{ -0.6 }{\cold}{\numd}
            \onesquareTwo{ 0.6 }{ -0.6 }{\cold}{\numd}
            \onesquareTwo{ 1.2 }{ -0.6 }{\cold}{\numd}
            \onesquareTwo{ 1.8 }{ -0.6 }{\cold}{\numd}
        \onesquareTwo{ 2.4 }{ -0.6 }{\cole}{\nume}
            \onesquareTwo{ 3.0 }{ -0.6 }{\cole}{\nume}
            \onesquareTwo{ 3.6 }{ -0.6 }{\cole}{\nume}
            \onesquareTwo{ 4.2 }{ -0.6 }{\cole}{\nume}
            \onesquareTwo{ 4.8 }{ -0.6 }{\cole}{\nume}

        \newcommand{\vgap}{3}
        
        \onesquareTwo{ 0.0 }{ 0.0 -\vgap}{\cola}{\numa}
            \onesquareTwo{ 0.6 }{ 0.0 -\vgap}{\cola}{\numa}
        \onesquareTwo{ 1.2 }{ 0.0 -\vgap}{\colb}{\numb}
            \onesquareTwo{ 1.8 }{ 0.0 -\vgap}{\colb}{\numb}
        \onesquareTwo{ 2.4 }{ 0.0 -\vgap}{\colc}{\numc}
            \onesquareTwo{ 3.0 }{ 0.0 -\vgap}{\colc}{\numc}
            \onesquareTwo{ 3.6 }{ 0.0 -\vgap}{\colc}{\numc}
        \onesquareTwo{ 4.2 }{ 0.0-\vgap }{\colGG}{\;}
        \onesquareTwo{ 0.0 }{ -0.6 -\vgap}{\cold}{\numd}
            \onesquareTwo{ 0.6 }{ -0.6 -\vgap}{\cold}{\numd}
            \onesquareTwo{ 1.2 }{ -0.6 -\vgap}{\cold}{\numd}
            \onesquareTwo{ 1.8 }{ -0.6 -\vgap}{\cold}{\numd}
        \onesquareTwo{ 2.4 }{ -0.6 -\vgap}{\cole}{\nume}
            \onesquareTwo{ 3.0 }{ -0.6 -\vgap}{\cole}{\nume}
            \onesquareTwo{ 3.6 }{ -0.6 -\vgap}{\cole}{\nume}
            \onesquareTwo{ 4.2 }{ -0.6 -\vgap}{\cole}{\nume}
            \onesquareTwo{ 4.8 }{ -0.6 -\vgap}{\cole}{\nume}

        \onesquareTwo{ 0.0 }{ 0.0 -2*\vgap }{\cola}{\numa}
            \onesquareTwo{ 0.6 }{ 0.0 -2*\vgap }{\cola}{\numa}
        \onesquareTwo{ 1.2 }{ 0.0 -2*\vgap}{\colb}{\numb}
            \onesquareTwo{ 1.8 }{ 0.0 -2*\vgap}{\colb}{\numb}
        \onesquareTwo{ 2.4 }{ 0.0 -2*\vgap}{\cold}{\numd}
            \onesquareTwo{ 3.0 }{ 0.0 -2*\vgap}{\cold}{\numd}
            \onesquareTwo{ 3.6 }{ 0.0 -2*\vgap}{\cold}{\numd}
            \onesquareTwo{ 4.2 }{ 0.0 -2*\vgap}{\cold}{\numd}
        \onesquareTwo{ 0.0 }{ -0.6 -2*\vgap}{\colc}{\numc}
            \onesquareTwo{ 0.6 }{ -0.6 -2*\vgap}{\colc}{\numc}
            \onesquareTwo{ 1.2 }{ -0.6 -2*\vgap}{\colc}{\numc}
        \onesquareTwo{ 1.8 }{ -0.6 -2*\vgap}{\cole}{\nume}
            \onesquareTwo{ 2.4 }{ -0.6 -2*\vgap}{\cole}{\nume}
            \onesquareTwo{ 3.0 }{ -0.6 -2*\vgap}{\cole}{\nume}
            \onesquareTwo{ 3.6 }{ -0.6 -2*\vgap}{\cole}{\nume}
            \onesquareTwo{ 4.2 }{ -0.6-2*\vgap }{\cole}{\nume}

        \draw[->, thick] (2, -1.2) -- (2, -0.9 - \vgap/2);

         \draw[->, thick] (2, -1.2 - \vgap) -- (2, -0.9 - \vgap/2 -\vgap);

        \node (init) at (-2.5, -0.3) {(Initial Placement)};

         \node (exp) at (-2.5, -0.3 - \vgap) {(Expansion Step)};

          \node (swap) at (-2.5, -0.3 - 2*\vgap) {(Final Swap)};
    \end{tikzpicture}
    \caption{Illustration of Case 1 for the proof of~\Cref{thm:2xn2}, for $n = 4$. In this case $\gamma_n = 1$.}
    \label{fig:case-1-proof}
\end{figure}



\subparagraph*{(\textbf{Case 2:} $\frac{n}{2} < \gamma_n < n - 1$)}

As before, placing the first $n-1$ rectangles in row one makes the first row $\frac{n}{2}$ cells too short. Then, if we place rectangles $r_n, \ldots, r_K$ in row two, given that in total $\gamma_n$ expansions are required to achieve the total desired area (\Cref{lemma:gap-def}), it must be the case that row two is $\gamma_n - \frac{n}{2}$ cells too short.  Naively, we would simply expand $\frac{n}{2}$ rectangles in the first row, and $\gamma_n - \frac{n}{2}$ in the second row. However, the second row might contain fewer than $\gamma_n - \frac{n}{2}$ rectangles. To address this, we will transfer a rectangle from row one to row two, and perform more expansions on row one, which concentrates most of the rectangles.
Let us identify which rectangle will be moved from row one to row two. Let us define
\[i = \gamma_n - \frac{n}{2}.\]
Transfer $r_i$ from row one to row two, and expand the first $\gamma_n < n-1$ of the rectangles in row one. Since $\gamma_n$ expansions have been made, the total area is exactly $\frac{n^2}{2}$, and thus it only remains to argue that the top row has exactly $\frac{n^2}{2}$ cells covered. This is indeed the case as 
\[
\frac{n(n-1)}{2} + \gamma_n - i = \frac{n(n-1)}{2} + \gamma_n - \left(\gamma_n - \frac{n}{2}\right) = \frac{n^2}{2}.
\]

\subparagraph*{(\textbf{Case 3:} $\gamma_n \geq n - 1$)}

Place the first $n-1$ rectangles in row one and the rest in row two. Expanding all $n - 1$ rectangles in the first row, and then expand $\gamma_n - (n-1)$ rectangles in the second row. Let $i = \frac{n}{2} - 2$ (if $n \in \{2, 4\}$, the result can be checked manually, and therefore we assume $i \geq 1$ is a valid index for a rectangle). Move rectangle $r_i$ from row one to row two. As in the previous case, it only remains to argue that the number of cells in the top row is exactly $\frac{n^2}{2}$. This is indeed the case as the number is determined by 
\[
\frac{n(n-1)}{2} + (n-1) - (i+1) = \frac{n(n-1)}{2} + (n-1) - \left(\frac{n}{2} - 1 \right) = \frac{n^2}{2}.
\]

Having covered all cases, we conclude the entire proof.




\end{proof}







\section{Complexity Results}
\label{sec:complexity}
In turn $t$ of a game of~\packit, the turn in which each of the already placed rectangles was packed into the grid is irrelevant, and therefore a \emph{partially filled grid $G$} of dimensions $n \times n$ can be represented as an $n \times n$ matrix over $\{0, 1\}$.  We will assume this representation uses $O(n^2)$ bits.
Consider now the following problem:
\vspace{8pt}
\newcommand{\solitairep}{Solitaire\packit}
\csproblem{\solitairep}{A partially filled grid $G$, and a turn number $t$ given in binary.}{Whether it is possible to complete a perfect packing for $G$ starting from turn $t$.}
\vspace{8pt}
We will analyze the complexity of~\solitairep~next, but before that, let us remark that the definition of the problem does not require the partial filling of $G$ to be achievable in $t-1$ turns. We leave the complexity of~\solitairep~with the additional restriction that $G$ must be achievable in $t-1$ turns as an open problem. That being said, we can present our main complexity result.

\newcommand{\np}{$\mathrm{NP}$}

\begin{theorem}
    \solitairep ~is~\np-complete.\label{thm:np-completeness}
\end{theorem}

\begin{proof}
Let $n \times n$ be the the dimensions of $G$. Membership in~\np~is easy to see: the certificate is a description of the turns $t, ..., t+m$, where $m = K(n, n) \leq  n^2$, and it suffices to check that at each turn $t + i$, a rectangle of the appropriate area was placed without overlapping with any of the previously placed rectangles. For hardness, we reduce from  a variant of the well-known 3 partition problem, proven to be NP-hard by Hulett, Will and Woeginger~\cite{hulettMultigraphRealizationsDegree2008}. The overall reduction is inspired by the analysis of Tetris by Breukelaar et al.~\cite{breukelaarTETRISHARDEVEN2004}. Consider the Restricted-3-Partition problem defined as follows.

\csproblem{Restricted-3-Partition}{A set of integers, $\mathcal{A} = \{\alpha_1, \ldots, \alpha_n\}$, with $n$ a multiple of $3$, such that if we define $T := \frac{\sum_{i=1}^n \alpha_i}{(n/3)}$, then $T/4 < \alpha_i < T/2$ for every $i \in [n]$.}{Whether it is possible to partition $\mathcal{A}$ into $n/3$ sets of $3$ elements, all of them having sum exactly $T$.}

Consider now the $4$-Restricted-3-Partition, defined exactly as above but with the additional restriction that all numbers $\alpha_i$ are multiples of $4$. This additional restriction preserves NP-hardness as every 3-partition $P$ defined as
\[
\left\{4\alpha_1, \ldots, 4\alpha_n\right\} \xmapsto{P} \left(\left\{4\alpha_i, 4\alpha_j, 4\alpha_k\right\},
\ldots,\left\{4\alpha_x, 4\alpha_y, 4\alpha_z\right\}\right)
\]
is in a one-to-one correspondence with a 3-partition $P'$ defined as 
\[
\{\alpha_1, \ldots, \alpha_n\} \xmapsto{P'} \left(\left\{\alpha_i, \alpha_j, \alpha_j\right\},
\ldots, \left\{\alpha_x, \alpha_y, \alpha_z\right\}\right).
\]
We can therefore reduce directly from $4$-Restricted-3-Partition.
Let $\mathcal{A}$ be an input instance of $4$-Restricted-3-Partition. We now show how to construct an associated instance of~\solitairep.
First, we will present the required \emph{gadgets}, which are illustrated in~\Cref{fig:gadgets}. 


\paragraph{\textbf{E-gadgets}. An $E$-gadget consists of a $T \times 3$ grid, in which the first and third column are completely filled, whereas the middle column is completely empty (hence $E$(mpty)-gadget). An illustration is presented in~\Cref{fig:s-gadget}.}


\paragraph{\textbf{S-gadgets}. Given an integer $\alpha \geq 1$, an $S(\alpha)$-gadget consists of a $T \times 3$ grid, in which the first and third column are completely filled, whereas only the bottom $T-\alpha$ rows of the middle column are filled. In other words, $S(\alpha)$-gadgets have a \emph{single} ``hole'' of $\alpha \times 1$, hence their name. An illustration is presented in~\Cref{fig:s-gadget}.}


\paragraph{\textbf{D-gadgets}. Given an integer $\alpha \geq 1$, a $D(\alpha)$-gadget consists of a $T \times 3$ grid, in which the first and third column are completely filled, and the middle column is filled only at row $\alpha+1$ and rows $\{2\alpha+2, 2\alpha+3, \ldots, T\}$. In other words, $D(\alpha)$-gadgets have two ``holes'' of $\alpha \times 1$, i.e., a \emph{double} hole, hence their name. An illustration is presented in~\Cref{fig:d-gadget}.}

\begin{figure}[t]
    \begin{subfigure}{0.3\linewidth}
    \centering
    \begin{tikzpicture}
      \foreach \i in {0,...,9} { 
            \onesquareTwo{ 0.0 }{ 0.0  + \i*0.6}{\colc}{}
      }

      \foreach \i in {0,...,9} { 
            \onesquareTwo{ 0.6 }{ 0.0  + \i*0.6}{\colGG}{}
      }

      \foreach \i in {0,...,9} { 
            \onesquareTwo{ 1.2 }{ 0.0  + \i*0.6}{\colc}{}
      }
        
    \end{tikzpicture}
    \label{fig:e-gadget}
    \caption{An $E$-gadget.}
    \end{subfigure}
    \begin{subfigure}{0.3\linewidth}
    \centering
    \begin{tikzpicture}
      \foreach \i in {0,...,9} { 
            \onesquareTwo{ 0.0 }{ 0.0  + \i*0.6}{\cole}{}
      }

      \foreach \i in {0,...,5} { 
            \onesquareTwo{ 0.6 }{ 0.0  + \i*0.6}{\cole}{}
      }

      \foreach \i in {6,...,9} { 
            \onesquareTwo{ 0.6 }{ 0.0  + \i*0.6}{\colGG}{}
      }

      \foreach \i in {0,...,9} { 
            \onesquareTwo{ 1.2 }{ 0.0  + \i*0.6}{\cole}{}
      }
        
    \end{tikzpicture}
    \caption{An $S(4)$-gadget.}
    \label{fig:s-gadget}
\end{subfigure}
\begin{subfigure}{0.3\linewidth}
    \centering
    \begin{tikzpicture}
      \foreach \i in {0,...,9} { 
            \onesquareTwo{ 0.0 }{ 0.0  + \i*0.6}{\colb}{}
      }

     \onesquareTwo{ 0.6 }{ 0.0  + 6*0.6}{\colb}{}

      \foreach \i in {0,...,4} { 
            \onesquareTwo{ 0.6 }{ 0.0  + \i*0.6}{\colb}{}
      }

    \onesquareTwo{ 0.6 }{ 0.0  + 3*0.6}{\colGG}{}
    \onesquareTwo{ 0.6 }{ 0.0  + 4*0.6}{\colGG}{}
    \onesquareTwo{ 0.6 }{ 0.0  + 5*0.6}{\colGG}{}
     \onesquareTwo{ 0.6 }{ 0.0  + 7*0.6}{\colGG}{}
     \onesquareTwo{ 0.6 }{ 0.0  + 8*0.6}{\colGG}{}
     \onesquareTwo{ 0.6 }{ 0.0  + 9*0.6}{\colGG}{}

      \foreach \i in {0,...,9} { 
            \onesquareTwo{ 1.2 }{ 0.0  + \i*0.6}{\colb}{}
      }
        
    \end{tikzpicture}
    \caption{A $D(3)$-gadget.}
    \label{fig:d-gadget}
\end{subfigure}
\caption{Illustration of the gadgets for $T = 10$.}
\label{fig:gadgets}
\end{figure}


With these gadgets, we can now construct a $(T+n)\times (T+n)$ grid as follows. First, horizontally concatenate exactly $n/3$ identical $E$-gadgets. Next, concatenate an $S(1)$-gadget to the right of the current construction. Then, for every odd value $m$ such that  $3 \leq m  <\max(\mathcal{A})$,  concatenate a $D(m)$-gadget to the right of the current construction if $m-1 \not\in \mathcal{A}$, and instead an $S(m+1)$-gadget to the right of the current construction otherwise.

\newcommand{\wdd}{0.25}
\begin{figure}[t]
    \centering
    \begin{tikzpicture}
        \draw[black] (0, 0) -- (0, 8) -- (8, 8) -- (8, 0) -- cycle;

        \draw[black, fill=\colc] (0, 0) -- (0, 8) --
        (\wdd, 8) -- (\wdd, 0) -- cycle;

         \draw[black, fill=\colc] (0 + 2*\wdd, 0) -- (0+ 2*\wdd, 8) --
        (\wdd+ 2*\wdd, 8) -- (\wdd+ 2*\wdd, 0) -- cycle;

         \draw[black, fill=\colc] (0+ 3*\wdd, 0) -- (0+ 3*\wdd, 8) --
        (\wdd+ 3*\wdd, 8) -- (\wdd+ 3*\wdd, 0) -- cycle;

         \draw[black, fill=\colc] (0 + 2*\wdd+ 3*\wdd, 0) -- (0+ 2*\wdd+ 3*\wdd, 8) --
        (\wdd+ 2*\wdd+ 3*\wdd, 8) -- (\wdd+ 2*\wdd+ 3*\wdd, 0) -- cycle;

         \node (sd) at (2, 4) {$\ldots$};

        \draw[black, fill=\colc] (10*\wdd, 0) -- (10*\wdd, 8) --
        (11*\wdd, 8) -- (11*\wdd, 0) -- cycle;       

        \draw[black, fill=\colc] (12*\wdd, 0) -- (12*\wdd, 8) --
        (13*\wdd, 8) -- (13*\wdd, 0) -- cycle;

         \draw[black, fill=\colb] (13*\wdd, 0) -- (13*\wdd, 8) --
        (14*\wdd, 8) -- (14*\wdd, 0) -- cycle;
        \draw[black, fill=\colb] (15*\wdd, 0) -- (15*\wdd, 8) --
        (16*\wdd, 8) -- (16*\wdd, 0) -- cycle;
         \draw[black, fill=\colb] (14*\wdd, 6) -- (14*\wdd, 6.25) --
        (15*\wdd, 6.25) -- (15*\wdd, 6) -- cycle;
        \draw[black, fill=\colb] (14*\wdd, 0) -- (14*\wdd, 4) --
        (15*\wdd, 4) -- (15*\wdd, 0) -- cycle;
        
        \draw[black, fill=\colb] (16*\wdd, 0) -- (16*\wdd, 8) --
        (17*\wdd, 8) -- (17*\wdd, 0) -- cycle;
        \draw[black, fill=\colb] (18*\wdd, 0) -- (18*\wdd, 8) --
        (19*\wdd, 8) -- (19*\wdd, 0) -- cycle;
         \draw[black, fill=\colb] (17*\wdd, 5.5) -- (17*\wdd, 5.75) --
        (18*\wdd, 5.75) -- (18*\wdd, 5.5) -- cycle;
        \draw[black, fill=\colb] (17*\wdd, 0) -- (17*\wdd, 3.25) --
        (18*\wdd, 3.25) -- (18*\wdd, 0) -- cycle;

        \draw[black, fill=\cole] (19*\wdd, 0) -- (19*\wdd, 8) --
        (20*\wdd, 8) -- (20*\wdd, 0) -- cycle;
        \draw[black, fill=\cole] (21*\wdd, 0) -- (21*\wdd, 8) --
        (22*\wdd, 8) -- (22*\wdd, 0) -- cycle;
        \draw[black, fill=\cole] (20*\wdd, 0) -- (20*\wdd, 6.5) --
        (21*\wdd, 6.5) -- (21*\wdd, 0) -- cycle;

        \draw[black, fill=\colb] (22*\wdd, 0) -- (22*\wdd, 8) --
        (23*\wdd, 8) -- (23*\wdd, 0) -- cycle;
        \draw[black, fill=\colb] (24*\wdd, 0) -- (24*\wdd, 8) --
        (25*\wdd, 8) -- (25*\wdd, 0) -- cycle;
         \draw[black, fill=\colb] (23*\wdd, 6.125) -- (23*\wdd, 6.375) --
        (24*\wdd, 6.375) -- (24*\wdd, 6.125) -- cycle;
        \draw[black, fill=\colb] (23*\wdd, 0) -- (23*\wdd, 4.25) --
        (24*\wdd, 4.25) -- (24*\wdd, 0) -- cycle;

        \draw[black, fill=\cole] (27*\wdd, 0) -- (27*\wdd, 8) --
        (28*\wdd, 8) -- (28*\wdd, 0) -- cycle;
        \draw[black, fill=\cole] (29*\wdd, 0) -- (29*\wdd, 8) --
        (30*\wdd, 8) -- (30*\wdd, 0) -- cycle;
        \draw[black, fill=\cole] (28*\wdd, 0) -- (28*\wdd, 7.25) --
        (29*\wdd, 7.25) -- (29*\wdd, 0) -- cycle;

        \draw[black, fill=\colt] (30*\wdd, 0) -- (30*\wdd, 8) --
        (11.5, 8) -- (11.5, 0) -- cycle;

        \draw[black, fill=\colt] (0, 0) -- (0, -3) --
        (11.5, -3) -- (11.5, 0) -- cycle;
        
        \node (sdotr2) at (25.9/4, 4) {\tiny $\ldots$};
    
        \node (yl) at (-0.5, 4) {$T$};
        \node (yl2) at (-0.5, -1.5) {$n$};
        \node (xl) at (4, -3.5) {$T+n$};
    \end{tikzpicture}
    \caption{Illustration of the construction of $G_\mathcal{A}$ for~\Cref{thm:np-completeness}. Note that $n$ could be larger than $T$, and thus this figure is not necessarily in scale.}
    \label{fig:np-completeness}
\end{figure}

Afterwards, if the resulting grid has length $T \times T'$, we complete a $T \times (T+n)$ grid by concatenating a $T \times (T+n - T')$ completely filled grid to the right of the current construction. This is well-defined, meaning that $T' < T+n$, as we show next. First, consider that, as each gadget uses exactly $3$ columns, we have
\begin{align*}
    T' &= 3\cdot n/3 + 3 \cdot \left|\big\{ 3 \leq m < \max(\mathcal{A}) \mid m \text{ is odd}\big\}\right|\\
    &\leq n + 3\left\lceil \frac{\max(\mathcal{A})-3}{2} \right\rceil
    < n + 3\frac{\max(\mathcal{A})}{2}. 
\end{align*}

Next, consider that 
\[
T = \left(\sum_{\alpha \in \mathcal{A}} \alpha \right)/(n/3) \leq 3\max(\mathcal{A}).
\]

Then, as $\max(\mathcal{A})\leq T/2$ by the definition of $4$-Restricted-3-Partition, we have 
\[
T' \leq n + 3  \cdot \frac{\max(\mathcal{A})}{2} \leq n + 3 \cdot \frac{T}{4} <T + n.
\]
Finally, to go from the resulting $T \times (T + n)$ grid to a $(T+n) \times (T+n)$ grid it suffices to concatenate a completely filled $n \times (T+n)$ grid at the bottom of the previous grid.
This construction is illustrated in~\Cref{fig:np-completeness}. We are now ready to prove the correctness of our reduction. Let $G_\mathcal{A}$ be the $(T+n) \times (T+n)$ grid constructed by the above process.
\begin{lemma}
    The instance $(G_\mathcal{A}, 1)$ is a Yes-instance for~\solitairep~if and only if~$\mathcal{A}$ is a Yes-instance for 4-Restricted-3-Partition.
    \label{lemma:reduction}
\end{lemma}
\begin{proof}
($\impliedby$) Let us start with the backward direction since it is simpler. Assume there is a solution to the partition problem with sets $S_1, \ldots, S_{n/3}$, where each set has exactly $3$ elements and its sum is exactly $T$. Then, we can complete a perfect packing of $G$ as follows. On each turn $1 \leq t \leq \max(\mathcal{A})$:
\begin{itemize}
    \item[Case \textbf{I})] If $t \in \mathcal{A}$, then let $i$ be the index such that $t = \alpha_i$, and $j$ be the index of the set $S_j$ such that $\alpha_i \in S_j$. Then, on this turn we can place a rectangle of dimensions $t \times 1$ into the $j$-th $E$-gadget of $G_\mathcal{A}$. 
    \item[Case \textbf{II})] If $t = 4k$ for some positive integer $k$ but $t \not\in \mathcal{A}$, then by construction there is a $D(t+1)$-gadget, which can be filled by placing a $(t+1) \times 1$ rectangle on this turn, and a $(t+1) \times 1$ rectangle on the next turn.
    \item [Case \textbf{III})] If $t = 4k+1$ and $t-1 \in \mathcal{A}$, then by construction there is an $S(t+1)$-gadget, which can be filled by placing a $(t+1) \times 1$ rectangle on this turn. 
    
    \item [Case \textbf{IV})] If $t = 4k+1$ for some integer $k$, and $t-1 \not\in \mathcal{A}$, then this turn has been covered in Case \textbf{II}).
    
    \item [Case \textbf{V})] If $t = 4k + 2$ for some integer $k$, then by by construction there is a $D(t+1)$-gadget, which can be filled by placing a $(t+1) \times 1$ rectangle on this turn, and a $(t+1) \times 1$ rectangle on the next turn.

    \item [Case \textbf{VI})] If $t = 4k+3$ then this turn has been covered in Case \textbf{V}).
\end{itemize}
As a result of the turns of Case \textbf{I}), every $E$-gadget will be completely filled since by definition, if $\alpha_i, \alpha_k, \alpha_\ell \in S_j$, then $\alpha_i + \alpha_k + \alpha_\ell = T$. As there are exactly $n/3$ identical $E$-gadgets in $G_\mathcal{A}$, they will all be filled. Note as well that the gadgets used in every case are different. In particular, the only $S$-gadgets in the construction are for $t+1 = 4k+2$ with $t-1 \in \mathcal{A}$, which are all used by Case \textbf{III}). Similarly, all $D(m)$-gadgets for $m = 4k+1$ for some integer $k$ are used by Case \textbf{II}), whereas all $D(m)$-gadgets for $m = 4k+3$ are used by Case \textbf{V}). Given all gadgets are perfectly filled up, we have a perfect packing of $G_\mathcal{A}$.

($\implies$) For the forward direction, assume it is possible to perfectly pack the grid $G_\mathcal{A}$ starting from turn $1$. Let $G^P_\mathcal{A}$ be any perfect packing completing $G_\mathcal{A}$. Note immediately that by construction, every rectangle placed in $G^P_\mathcal{A}$ from turn $1$ onward must have dimension $t \times 1$ for some positive integer $t$. Intuitively, we will now prove that the choices made in the backward direction of the proof are forced.

\begin{definition}
    For any turn $t \geq 1$, we say the rectangle placed in $G^P_\mathcal{A}$ on turn $t$ is \emph{proper} if either 
    \begin{enumerate}
        \item $t = 1$, and the rectangle placed  in $G^P_\mathcal{A}$ on this turn was a $1 \times 1$ rectangle placed in the only $S(1)$-gadget of $G_\mathcal{A}$.
        \item $t > 1$ is odd, and $t-1 \in \mathcal{A}$, and the rectangle placed  in $G^P_\mathcal{A}$ on this turn was a $(t+1) \times 1$ placed in the only $S(t+1)$-gadget of $G_\mathcal{A}$.
        \item $t > 1$ is odd and $t-1 \not\in\mathcal{A}$, and the rectangle placed  in $G^P_\mathcal{A}$ on this turn was a $t \times 1$ placed in one of the two spaces of the only $D(t)$-gadget of $G_\mathcal{A}$.
        \item $t \in \mathcal{A}$, and the rectangle placed  in $G^P_\mathcal{A}$ on this turn was placed in one of the $E$-gadgets. 
        \item $t$ is even but $t \not\in \mathcal{A}$, and the rectangle placed  in $G^P_\mathcal{A}$ on this turn was a $(t+1) \times 1$ placed in one of the two spaces of the only $D(t+1)$-gadget of $G_\mathcal{A}$.
    \end{enumerate}
\end{definition}

\begin{claim}
    Every turn $t \geq 1$ where a rectangle was placed in $G^P_\mathcal{A}$ must have been proper. 
    \label{claim:proper}
\end{claim}
\begin{proof}[Proof of~\Cref{claim:proper}]
    We prove the claim by induction on $t$. The base case is $t=1$, for which a single $S(1)$-gadget exists in the construction, and given that the $1 \times 1$ empty space in this gadget must be filled in $G^P_\mathcal{A}$, the only turn on which it can be filled is turn $1$. Therefore the base case works. For the inductive case, 
    assume the claim holds up to $t$ and let us show it holds for $t+1$. 
    
    \begin{itemize}
        \item If $t+1$ is odd and $t \in \mathcal{A}$, then we claim the rectangle placed on turn $t+1$ must have been a $(t+2) \times 1$ rectangle in the only $S(t+2)$-gadget of $G_\mathcal{A}$.  Indeed, if this were not the case, said gadget could only have been filled by a $(t+2) \times 1$ rectangle placed on turn $(t+2)$, since all previous turns have been proper and thus not placed anything in the $S(t+2)$-gadget. However, given there are two empty spaces of size $(t+3)$ into the only $D(t+3)$-gadget of $G_\mathcal{A}$ (which must exist since $t \in \mathcal{A} \implies t+2 \not \in \mathcal{A}$ as all elements of  $\mathcal{A}$ are multiples of $4$), and no previous turns could have placed anything into them as they are proper by inductive hypothesis, then we conclude that on turn $(t+2)$ a rectangle of size $(t+3)$ must have been placed into the only $D(t+3)$-gadget of $G_\mathcal{A}$.

    \item If $t+1$ is odd and $t \not\in\mathcal{A}$, then given all the previous turns have been proper, it must be that the only $D(t+1)$-gadget of $G_\mathcal{A}$ has only received a $(t+1) \times 1$ rectangle placed on turn $t$, according to (5) in the definition of proper turn. Therefore, a single $(t+1) \times 1$ empty space remains in the only $D(t+1)$-gadget of $G_\mathcal{A}$, and it must be that is filled on this turn, as any posterior turns will have rectangles of area at least $t+2$.

    \item If $t+1$ is even but $t+1 \not\in \mathcal{A}$, then given all turns so far have been proper, there are two empty $(t+2) \times 1$ spaces in the only $D(t+2)$-gadget of $G_\mathcal{A}$, and given none can be filled after turn $t+3$, and at most one can be filled in turn $t+2$, we conclude that turn $t+1$ must fill one.
    \item If $t+1 \in \mathcal{A}$, and this turn were to be improper, then the rectangle placed on this turn must be placed either in an $S(t')$-gadget or in a $D(t')$-gadget.
    
    In either case we will reach a contradiction.
    Note first that $t' > t+2$: in the construction of $G_\mathcal{A}$, as $t$ is odd and $t-1 \not \in \mathcal{A}$, when $m = t$ a $D(t)$-gadget was created, and the next gadget created is a $S(t+3)$-gadget when $m = t+2$, since $m-1 \in \mathcal{A}$.
    Next, note that the remaining empty space on the $S(t')$-gadget or the $D(t')$-gadget partially filled on turn $t+1$ must be at least $t' - (t+2) > 0$. If $t' - (t+2) < t+2$, then that remaining empty space can never be filled in posterior turns, where all rectangles have area at least $t+2$, a contradiction. Otherwise, $t' - (t+2) > t+1$, meaning that $t' > 2t+3$. Because an $S(t')$-gadget or a $D(t')$-gadget exists, we deduce from the construction that $t' \leq \max(\mathcal{A})$. This implies that \[
    \max(\mathcal{A}) > t' - 1 > 2t + 2 = 2(t+1),
    \]
    meaning that two elements of $\mathcal{A}$, namely \[
    \alpha_i:= \max(\mathcal{A}), \quad \alpha_j := t+1, \]
 hold $\alpha_i > 2\alpha_j$.
    But by definition of $4$-Restricted-3-Partition that would imply the following contradiction:
     \[
     T/4 < \alpha_j < \alpha_i/2 < (T/2)/2 = T/4.
     \]\qedhere
     \end{itemize}
\end{proof}

By~\Cref{claim:proper}, we have that for every  $\alpha_t \in \mathcal{A}$, a rectangle of area $\alpha_t$ has been placed inside an $E$-gadget. Given that $T/2 < \alpha_t < T/4$ for every $t$, there must be exactly $3$ rectangles placed inside every $E$-gadget. Let $\alpha_i^{(1)}, \alpha_i^{(2)}, \alpha_i^{(3)}$ be the areas of the three rectangles placed inside the $i$-th $E$-gadget. As for every $i$, by hypothesis, the $i$-th
$E$-gadget is perfectly filled and had $t$ empty cells to be filled, we conclude that that $\alpha_i^{(1)} +\alpha_i^{(2)} + \alpha_i^{(3)}= T$, from where it follows that $\mathcal{A}$ is a Yes-instance to the $4$-Restricted-3-Partition problem. This concludes the proof of~\Cref{lemma:reduction}. 


\end{proof}

Given the reduction presented above can clearly be carried out in polynomial time, we conclude hardness from the correctness proved in~\Cref{lemma:reduction}, and consequently this finishes the entire proof of~\Cref{thm:np-completeness}.


\end{proof}

\section{Computing Perfect \packit~games}
\label{sec:computation}
Even though~\Cref{thm:np-completeness} does not directly imply that it is hard to find perfect packings for an $n \times n$ grid (or to decide whether such a packing exist), it arguably gives evidence for this being a hard combinatorial challenge. 

In many combinatorial problems SAT-solving
can dramatically outperform backtracking approaches. This also happens to be the case for computing perfect games of~\packit, where even after several optimizations, a backtracking approach only allowed us to find perfect packings up to $n=20$. In contrast, by using a novel SAT encoding technique we were able to find perfect packings up to $n=50$ in under 24 hours of computation. 
We assume a basic familiarity with propositional logic and SAT-solving, but
for a general reference we direct the reader to~\cite{Biere2021-od}. 
As in~\Cref{sec:results}, we divide the problem into two stages: (i) finding a set of rectangles $(h_t, v_t)$ such that 
\begin{itemize}
    \item Their total area is $n^2$, meaning that $\sum_{t} h_t \cdot v_t = n^2$.
    \item The $t$-th rectangle has area $t$ or $t+1$, meaning that $h_t \cdot v_t \in \{t, t+1\}$ for every $t$.
    \item All rectangles fit into the $n \times n$ grid, meaning that $\max(h_t, v_t) \leq n$.
\end{itemize}
and (ii), packing the rectangles obtained in the previous stage without overlaps.
Note that due to the area condition, if a valid rectangle selection is packed without overlapping, then they must cover the entire $n \times n$ grid. 

For stage (i), we use a pseudo-polynomial dynamic programming approach, similar to the one used for the standard subset sum problem. For stage (ii) we use a sophisticated SAT encoding that uses only $O(n^3)$ many clauses as opposed to the naive $O(n^4)$ encoding. 

%

\subsection{Rectangle Selection through Dynamic Programming} For stage (i), we use a pseudo-polynomial dynamic programming approach, similar to the one used for standard subset sum. Given integer values $(a, n)$, let $R(a, n)$ be a function returning $1$ if there is a rectangle $(h, v)$ such that $h \cdot v = a$ and $\max(h, v) \leq n$, and $0$ if no such rectangle exists. Clearly $R$ can be computed efficiently. Now, let us say that
a sequence of $t$ rectangles $(h_i, v_i)$, whose total area adds up to $T$, and such that $\max(h_i, v_i) \leq n$ is a valid $(t, T, n)$-rectangle-selection. Let us abbreviate $K(n, n)$ as $K$ from now on. Our goal is to compute a valid $(K, n^2, n)$-rectangle-selection. Let us define $\textsf{dp}(t, T)$ as 
\[
\textsf{dp}(t, T) := \begin{cases}
    1 & \text{if there is a valid $(t, T, n)$-rectangle-selection,}\\
    0 & \text{otherwise.}
\end{cases}
\]
We can compute $\textsf{dp}(t, T)$ efficiently through the following recursive equation:

\[
\textsf{dp}(t, T) = \begin{cases}
    1 & \text{if $t = T = 0$,}\\
    0 & \text{else if $t = 0 \text { or } T < 0$,}\\
    \max\big\{ \textsf{dp}(t-1, T-t) \cdot R(t, n) , & \text{ } \\ 
    \quad \; \; \; \; \; \, \textsf{dp}(t-1, T-t-1) \cdot R(t+1, n) \big\} & \text{otherwise.} 
\end{cases}
\]
By pre-computing the value of the function $R$, and then filling up a table up to $\textsf{dp}(K, n^2)$, this approach runs in time $O(n^2 \cdot K) = O(n^{3})$\footnote{The fact that $K(n, n) = O(n)$ is a direct consequence of footnote 1.}, which is almost instantaneous for the values of $n$ we will consider. Naturally, from the table $\textsf{dp}(t, T)$ we can easily reconstruct a valid $(K, n^2, n)$-rectangle-selection. 

\subsection{An efficient SAT encoding}
Assume now we have a valid $(K, n^2, n)$-rectangle-selection $L = \left( (h_1, v_1), \ldots, (h_K, v_K) \right)$. How can we pack these selection of rectangles into the $n \times n$ grid? This subsection presents an efficient encoding for this particular case of the NP-hard problem of rectangle packing~\cite[Theorem 6.19]{ComputationalIntractabilityGuideb}. 
We assume that each rectangle in the selection might need to be rotated by $90^\circ$,  as the dynamic programming generating the rectangle selection was not concerned with distinctions between the horizontal and vertical dimension.
A naive encoding can be designed as follows. For every index $t \in [K]$, create a boolean variable $r_t$ corresponding to whether the $t$-th rectangle in the selection $L$ will be rotated or not. Then, for any position $(i, j) \in \{0, \ldots, n-1 \} \times \{0, \ldots, n-1\}$, and index $t \in [K]$, create a boolean variable $p_{t, (i, j)}$ stating that the $t$-th rectangle will be placed with its top-left corner in position $(i, j)$. Then, to force each rectangle to be placed somewhere, one can add clauses 
\[
\bigvee_{(i, j) \text{ s.t.  rectangle $t$ would fit if placed at $(i, j)$}} p_{t, {i, j}}, \quad \forall t \in [K].
\]
In order to avoid overlaps, we would now need binary clauses of the form
\[
 \neg p_{t_1, (i, j)} \lor \neg p_{t_2, (a, b)}
\]
for every pair $t_1 \neq t_2$ and every quadruple of indices $i,j, a, b$ such that $t_1$ and $t_2$ would overlap if placed at $(i, j)$ and $(a, b)$ respectively. For the sake of simplicity, we are omitting the incorporation of the $r_t$ variables, as regardless of them, this encoding results in $O(n^2 \cdot K^2) = O(n^4)$ many clauses, which is prohibitively large even for very small values of $n$. 

Instead, we will build an encoding using only $O(n^3)$ many clauses. The key insight is that we can decompose the problem into a \emph{horizontal component} and a \emph{vertical component}, that can be treated separately and yet are connected in a last step where we establish that pairs of rectangles can overlap in their $x$-coordinates or in their $y$-coordinates, but not both. We proceed to present the encoding, although we remark immediately that despite its conceptual simplicity, the details are quite intricate, and we plan to write a subsequent paper exploring the details of this encoding.

For each rectangle index $t \in [K]$, we will create variables 
\(
x_{t, i}
\) for each $0 \leq i < n$, and similarly, \(
y_{t, j}
\) for each $0 \leq j < n$. Intuitively, $y_{t, j}$ represents that some cell of the $t$-th rectangle will use the $i$-th  row of the grid, and $x_{t, i}$ that some cell will use the $j$-th column of the grid.

For any $t \in [K]$, let us denote by $\vec{x}_t$ the vector of boolean variables $(x_{t, 0}, \ldots, x_{t, n-1})$, and analogously for $\vec{y}_t$.
The high level requirement on the $x_{t, i}, y_{t, j}$ variables will be that if the $t$-th rectangle had dimensions $(h_t, v_t)$, then $\vec{x}_t$ must be a vector of $0$s\footnote{We use the traditional assumption of boolean variables taking values in $\{0, 1\}$.}, except for a subsegment of exactly $v_t$ consecutive values being $1$s, and $\vec{y}_t$ must be a vector of $0$s except for a subsegment of exactly $h_t$ consecutive $1$s. Depending on the value of the variables $r_t$, these roles will be inverted. Interestingly, these requirements can be implemented with $O(n)$ clauses and $O(n)$ extra variables for each index $t \in [K]$ as we show next.

Consider without loss of generality that we focus on the horizontal dimension of the problem, as the vertical direction is analogous. Then, create variables $m^x_{t, i}$ and $M^x_{t, i}$ for every $t \in [K]$ and $0 \leq i < n$. The clauses involving $m^x_{t, i}$ and $M^x_{t, i}$ will be such that if rectangle $t$ uses columns $\ell, \ell+1, \ldots, r$, then we will have
\newcommand{\mvecxt}{\overrightarrow{m^x_t}}
\newcommand{\Mvecxt}{\overrightarrow{M^x_t}}

\[
\mvecxt = (1, \ldots, 1, \ldots, 1, 0, \ldots, 0),
\]
where the prefix of $1s$ in $\mvecxt$ has length $r$, and
\[
\Mvecxt = (0, \ldots, 0, 1,\ldots, 1, \ldots, 1),
\]
where the first one in $\Mvecxt$ appears in position $\ell$. The idea now is that $\vec{x}_t$ should have $1$s exactly in the positions where both $\mvecxt$ and $\Mvecxt$ have $1$s, as illustrated in~\Cref{fig:encoding-illustration}. 

The main question is how to encode this desired behavior of the variables $\mvecxt$ and $\Mvecxt$ can be achieved, as the values of $\vec{x}_t$ can be established by simply adding the clausal form of
\[
    x_{t, i} \iff \left(m^{x}_{t, i} \land M^{x}_{t, i}\right),
\]
for every $t \in [K], i \in \{0, \ldots, n-1\}$. In order to obtain the desired behavior for the variables $\mvecxt$ and $\Mvecxt$, we can use the following three ideas:
\begin{enumerate}
    \item As every $1$ in $\Mvecxt$ is followed by a $1$, it suffices to add the implications $M^{x}_{t, i} \implies M^{x}_{t, i+1}$ for every $i \in \{0, \ldots, n-2\}$.
    \item As every $1$ in $\mvecxt$ is preceded by a $1$, it suffices to add the implications $m^{x}_{t, i} \implies M^{x}_{t, i-1}$ for every $i \in \{1, \ldots, n-1\}$.
    \item Based on the previous two steps, it must be the case that in any satisfying assignment to the formula created thus far there is a value $R \in \{0, \ldots, n-1\}$ such that 
    \[ m^x_{t, i} \iff i \leq R,\] and a value $L \in \{0, \ldots, n-1\}$ such that \[
    M^x_{t, i} \iff i \geq L.
    \]
    In order to enforce that there is exactly $h_t$ positions $i$ in which both $m^x_{t, i}$ and $M^{x}_{t, i}$ have 1s, it suffices to enforce that $R - L = h_t-1$. To achieve this, for each $i$, add in clausal form the implications:
    \[
    m^x_{t, i} \implies \neg M^x_{t, i - h_t}, \quad \forall i \text{ s.t. } i - h_t \geq 0,
    \]
    as well as 
    \[
    M^x_{t, i} \implies \neg m^x_{t, i + h_t}, \quad \forall i \text{ s.t. } i + h_t < n.
    \]
    Now observe that from the previous implications one derives the following chain of reasoning:
    \begin{align*}
        i \leq R &\iff m^x_{t, i}\\ 
        &\iff \neg M^x_{t, i-h_t}\\
        &\iff i - h_t < L\\
        & \iff i < L + h_t,
    \end{align*}
from where $L + h_t < R$ and similarly,
  \begin{align*}
        i \geq L &\iff M^x_{t, i}\\ 
        &\iff \neg m^x_{t, i+h_t}\\
        &\iff i + h_t > R\\
        & \iff i > R - h_t,
    \end{align*}
from where $L > R - h_t$. As all values are integers, we have 
\begin{align*}
    (L + h_t < R) \land (L > R - h_t) &\iff (L + h_t \leq R+1) \land (L-1 \geq R - h_t) \\
    &\iff (h_t - 1\leq R-L) \land ( h_t- 1 \geq R - L )\\
    & \iff R - L = h_t -1,
\end{align*}
thus implying the desired result.
    
\end{enumerate}

\begin{figure}[ht]
    \centering
    \begin{tikzpicture}
        \onesquare{ 0.0 }{ 3.75 }{\colGG}{\;}
\onesquare{ 0.75 }{ 3.75 }{\colGG}{\;}
\onesquare{ 1.5 }{ 3.75 }{\colGG}{\;}
\onesquare{ 2.25 }{ 3.75 }{\colGG}{\;}
\onesquare{ 3.0 }{ 3.75 }{\colGG}{\;}
\onesquare{ 3.75 }{ 3.75 }{\colGG}{\;}
\onesquare{ 0.0 }{ 3.0 }{\colGG}{\;}
\onesquare{ 0.75 }{ 3.0 }{\colb}{\numb}
\onesquare{ 1.5 }{ 3.0 }{\colb}{\numb}
\onesquare{ 2.25 }{ 3.0 }{\colb}{\numb}
\onesquare{ 3.0 }{ 3.0 }{\colGG}{\;}
\onesquare{ 3.75 }{ 3.0 }{\colGG}{\;}

\node  (mlabel) at (-1, 2) {$\overrightarrow{m^x_2}$};

\node (mvals) at (1.85, 2) {$\mathbf{1}, \;\;\;\,  \mathbf{1}, \;\;\;\,  \mathbf{1}, \;\;\;\,  \mathbf{1}, \;\;\;\, 0, \;\;\;\, 0$};

\node  (xlabel) at (-1, 1) {$\overrightarrow{x_2}$};

\node (xvals) at (1.85, 1) {$0, \;\;\;\,  \mathbf{1}, \;\;\;\,  \mathbf{1}, \;\;\;\,  \mathbf{1}, \;\;\;\, 0, \;\;\;\, 0$};

\node  (Mlabel) at (-1, 4.75) {$\overrightarrow{M^x_2}$};

\node (Mvals) at (1.85, 4.75) {$0, \;\;\;\,  \mathbf{1}, \;\;\;\,  \mathbf{1}, \;\;\;\,  \mathbf{1}, \;\;\;\, \mathbf{1}, \;\;\;\, \mathbf{1}$};

    \end{tikzpicture}
    \caption{Example for the auxiliary variables in the efficient SAT encoding, where $t=2$, and a rectangle $(h_2 := 3, v_2 := 1)$ is considered. }
    \label{fig:encoding-illustration}
\end{figure}


For the variables $\vec{y}_t$ we proceed analogously. Then, we complete the encoding by forbidding overlaps as follows.
For each pair $t_1, t_2$ of rectangles, create an auxiliary variable $c_{t_1, t_2}$. Intuitively, we will use $c_{t_1, t_2}$ to force $t_1$ and $t_2$ to not overlap. Then, for each pair of rectangles $t_1, t_2$, create clauses representing the following implications:
\begin{equation}
   \left( \bigvee_{i = 0}^{n-1} \left(x_{t_1, i} \land x_{t_2, i} \right) \right) \implies c_{t_1, t_2},\label{eq:x-inter}
\end{equation}
\begin{equation}
    \left( \bigvee_{i = 0}^{n-1} \left(y_{t_1, i} \land y_{t_2, i} \right) \right) \implies \neg c_{t_1, t_2},\label{eq:y-inter}
\end{equation}
which requires $O(n)$ many clauses for each pair of rectangles. This amounts to a total of $O(nK^2) = O(n^3)$ clauses in total. Correctness can be argued as follows.~\Cref{eq:x-inter} implies that if rectangles $t_1$ and $t_2$ intersect on an $x$-coordinate, then $c_{t_1, t_2}$ must be true, whereas~\Cref{eq:y-inter} implies that if they intersect on a $y$-coordinate, then $c_{t_1, t_2}$ must be false. As $c_{t_1, t_2}$ cannot be simultaneously true and false, it cannot be the case that $t_1$ and $t_2$ intersect on both an $x$ and a $y$ coordinate, which implies they cannot overlap. This concludes the general description of the encoding.

\subsection{Computational Results}
All experiments have been run on a personal computer with the following specifications:

\begin{itemize}
    \item MacBook Pro M1, 2020, running Sonoma 14.3
    \item 16GB of RAM
    \item 8 cores (but all experiments were run in a single thread).
\end{itemize}

In terms of software, we experimented with different SAT-solvers, and obtained the best results using the award-winning solver~\textsf{Kissat}~\cite{biereCaDiCaLKissatParacooba2020}.
We tested every value of $n$ between $5$ and $50$ and such that neither~\Cref{thm:small-gap} nor~\Cref{thm:large-gap} applies, and for every value we were able to find a perfect game of~\packit~in under 24 hours. For each such value, we used the dynamic programming approach to generate a valid selection of rectangles, and simply used the first one obtained. Given the number of valid selections of rectangles is likely exponential in $n$, it could be that some valid selections are significantly easier to pack than others. The fact that we obtained perfect packings simply using the first valid rectangle selection obtained via dynamic programming confirms the robustness of the SAT approach.

Detailed results are presented in~\Cref{tab:sat-results}.

\begin{table}[ht!]
\caption{Computational results for $n \in \{5, \ldots, 50\}$. Perfect packings for $n \in \{1, \ldots, 4\}$ are trivial.}
    \centering
    \scalebox{0.85}{
    \begin{tabular}{rrrr}
    \toprule
$n$ & \#vars & \#clauses & SAT runtime \\
\midrule
5 & 141 & 424 & 0.0s \\ 
(\Cref{thm:small-gap} applies)  6 &  - &  - & - \\ 
7 & 297 & 1101 & 0.0s \\ 
8 & 375 & 1482 & 0.0s \\ 
9 & 510 & 2228 & 0.02s \\ 
10 & 611 & 2797 & 0.02s \\ 
11 & 780 & 3921 & 0.02s \\ 
12 & 904 & 4732 & 0.03s \\ 
13 & 1037 & 5673 & 0.19s \\ 
14 & 1254 & 7375 & 0.16s \\ 
15 & 1410 & 8584 & 0.04s \\ 
16 & 1661 & 10838 & 0.56s \\ 
17 & 1840 & 12397 & 0.20s \\ 
(\Cref{thm:large-gap} applies) 18 & - & - & - \\ 
19 & 2327 & 17184 & 0.20s \\ 
20 & 2538 & 19339 & 2.47s \\ 
21 & 2871 & 23037 & 2.08s \\ 
22 & 3105 & 25582 & 2.04s \\ 
(\Cref{thm:small-gap} applies) 23 & - & - & - \\ 
24 & 3729 & 33117 & 4.43s \\ 
25 & 3995 & 36396 & 2.80s \\ 
26 & 4410 & 41980 & 2.69s \\ 
27 & 4699 & 45737 & 23.21s \\ 
28 & 5148 & 52283 & 8.45s \\ 
29 & 5460 & 56636 & 17.24s\\ 
(\Cref{thm:large-gap} applies) 30 & -  & - & - \\ 
31 & 6278 & 69109 & 34.26s \\ 
32 & 6622 & 74340 & 48.17s \\ 
33 & 7153 & 83288 & 36.37s \\ 
34 & 7520 & 89207 & 107.23s \\ 
(\Cref{thm:small-gap} applies) 35 & -  & - & - \\ 
36 & 8475 & 105934 & 747.46s \\ 
37 & 8874 & 112997 & 194.33s \\ 
38 & 9487 & 124629 & 502.20s \\ 
39 & 9909 & 132324 & 442.62s \\ 
40 & 10556 & 145392 & 129.71s \\ 
41 & 11001 & 153969 & 6117.58s \\ 
42 & 11455 & 162890 & 2088.45s \\ 
43 & 12150 & 177744 & 923.03s \\ 
44 & 12627 & 187501 & 579.50s \\ 
45 & 13356 & 203857 & 3185.11s \\ 
46 & 13856 & 214540 & 2188.39s \\ 
(\Cref{thm:large-gap} applies) 47 & -  & - & - \\ 
48 & 15142 & 244107 & 48102.44s \\ 
49 & 15674 & 256188 & 23337.97s \\ 
50 & 16485 & 276182 & 15925.77s \\ 
\bottomrule
    \end{tabular}
}
    \label{tab:sat-results}
\end{table}

\section{Concluding Remarks}
\label{sec:conclusion}
We have analyzed several aspects of~\packit, with our main result being (i) conditions for perfect games provided by~\Cref{thm:small-gap,thm:large-gap}, (ii) every $2 \times \frac{n^2}{2}$ grid admits a perfect~\packit~game, and (iii) for every $n \leq 50$ such that neither~\Cref{thm:small-gap}~nor~\Cref{thm:large-gap} applies, the $n \times n$ grid admits a perfect~\packit~game. In other words, ~\Cref{conjecture:main}  is true for all values of $n \leq 50$. 

We hope that both our mathematical and computational techniques can be applicable to similar packing problems. The \emph{``Mondrian Art Puzzle''}~\cite{garcia2023there, okuhn2018mondrian} asks for perfect packings of $n\times n$ grids but where all rectangles must use the same area. Recently, the \emph{MIT CompGeom Group} has studied perfect packings for rectangular grids with square pieces~\cite{mitcompgeomgroup2023tile}. Then, in terms of concrete~\packit~questions, we pose the following challenges:
\begin{enumerate}
    \item Prove or refute~\Cref{conjecture:main}.
    \item Is there always a perfect packing of the $m \times n$ grid when $\gamma(m, n) = K(m, n)/2$? In this case, exactly half of the turns are expansion turns. In particular, this might be easier to show assuming $m$ and $n$ are even.
    \item What is the complexity of~\packit~as a 2-player game? It is well known that complexity tends to increase in 2-player formulations (see e.g.,~\cite{ComputationalIntractabilityGuideb}), so could~\packit~be complete for the class~\textsf{PSPACE}?
\end{enumerate}


\bibliography{main}

\begin{thebibliography}{10}

\bibitem{barbay_et_al:LIPIcs.FUN.2021.23}
J\'{e}r\'{e}my Barbay and Bernardo Subercaseaux.
\newblock {The Computational Complexity of Evil Hangman}.
\newblock In Martin Farach-Colton, Giuseppe Prencipe, and Ryuhei Uehara,
  editors, {\em 10th International Conference on Fun with Algorithms (FUN
  2021)}, volume 157 of {\em Leibniz International Proceedings in Informatics
  (LIPIcs)}, pages 23:1--23:12, Dagstuhl, Germany, 2020. Schloss Dagstuhl --
  Leibniz-Zentrum f{\"u}r Informatik.

\bibitem{Barbeau2003-tc}
Edward~J Barbeau.
\newblock {\em Pell's Equation}.
\newblock Problem Books in Mathematics. Springer, New York, NY, 2003 edition,
  January 2003.

\bibitem{biereCaDiCaLKissatParacooba2020}
Armin Biere, Katalin Fazekas, Mathias Fleury, and Maximillian Heisinger.
\newblock {{CaDiCaL}}, {{Kissat}}, {{Paracooba}}, {{Plingeling}} and
  {{Treengeling Entering}} the {{SAT Competition}} 2020.
\newblock In Tomas Balyo, Nils Froleyks, Marijn Heule, Markus Iser, Matti
  J{\"a}rvisalo, and Martin Suda, editors, {\em Proc. of {{SAT Competition}}
  2020 {\textendash} {{Solver}} and {{Benchmark Descriptions}}}, volume
  B-2020-1 of {\em Department of {{Computer Science Report Series B}}}, pages
  51--53. {University of Helsinki}, 2020.

\bibitem{boutonNimGameComplete1901}
Charles~L. Bouton.
\newblock Nim, {{A Game}} with a {{Complete Mathematical Theory}}.
\newblock {\em Annals of Mathematics}, 3(1/4):35--39, 1901.

\bibitem{breukelaarTETRISHARDEVEN2004}
{\relax Ron}~Breukelaar, Erik~D. Demaine, {\relax Susan}~Hohenberger,
  Hendrik~Jan Hoogeboom, Walter~A. Kosters, and {\relax Dadvid}~{Liben-Nowell}.
\newblock {{Tetris is Hard}}, {{Even to Approximate}}.
\newblock {\em International Journal of Computational Geometry \&
  Applications}, 14:41--68, April 2004.

\bibitem{buchin2021dots}
Kevin Buchin, Mart Hagedoorn, Irina Kostitsyna, and Max van Mulken.
\newblock Dots \& boxes is {PSPACE}-complete, 2021.

\bibitem{ComputationalIntractabilityGuideb}
Erik.~D Demaine, William Gasarch, and Mohammad Hajiaghayi.
\newblock Computational {{Intractability}}: {{A Guide}} to {{Algorithmic Lower
  Bounds}}.
\newblock https://hardness.mit.edu/.

\bibitem{dusar1998}
Pierre Dusart.
\newblock {\em Autour de la fonction qui compte le nombre de nombres premiers}.
\newblock PhD thesis, Universit\'e de Limoges, 1998.
\newblock Thèse de doctorat dirigée par Robin, Guy Mathématiques
  appliquées. Théorie des nombres Limoges 1998.

\bibitem{garcia2023there}
Natalia Garc{\'\i}a-Col{\'\i}n, Dimitri Leemans, Mia M{\"u}{\ss}ig, and
  {\'E}rika Rold{\'a}n.
\newblock There is no perfect mondrian partition for squares of side lengths
  less than 1001.
\newblock {\em arXiv preprint arXiv:2311.02385}, 2023.

\bibitem{Gardner1970}
Martin Gardner.
\newblock Mathematical games.
\newblock {\em Scientific American}, 223(4):120–123, October 1970.

\bibitem{mitcompgeomgroup2023tile}
MIT~CompGeom Group, Zachary Abel, Hugo~A. Akitaya, Erik~D. Demaine, Adam~C.
  Hesterberg, and Jayson Lynch.
\newblock When can you tile an integer rectangle with integer squares?, 2023.

\bibitem{hulettMultigraphRealizationsDegree2008}
Heather Hulett, Todd~G. Will, and Gerhard~J. Woeginger.
\newblock Multigraph realizations of degree sequences: {{Maximization}} is
  easy, minimization is hard.
\newblock {\em Operations Research Letters}, 36(5):594--596, September 2008.

\bibitem{mcguire201316clue}
Gary McGuire, Bastian Tugemann, and Gilles Civario.
\newblock There is no 16-clue sudoku: Solving the sudoku minimum number of
  clues problem, 2013.

\bibitem{okuhn2018mondrian}
Cooper O'Kuhn.
\newblock The mondrian puzzle: A connection to number theory, 2018.

\bibitem{schoenfeldSharperBoundsChebyshev1976}
Lowell Schoenfeld.
\newblock Sharper {{Bounds}} for the {{Chebyshev Functions}} $\theta(x)$ and
  $\psi(x)$. {{II}}.
\newblock {\em Mathematics of Computation}, 30(134):337--360, 1976.

\end{thebibliography}

\bibliographystyle{plain}
\newpage
\appendix

\paragraph{\textbf{Acknowledgments.}}
The authors are partially supported by the National Science Foundation (NSF) grant CCF-2108521. We thank FUN'2024 reviewers for their feedback and suggestions. We also thank Richard Green, for his comments and his blog post about our paper. The last author thanks Abigail Kamenenetsky for her help with a web implementation of the game.

\section{Selected perfect games of~\packit}

\input{fig-24}

\input{fig-36}

\input{fig-50}







\end{document}